\documentclass[12pt]{amsart}
\usepackage{amsmath, amsthm, amssymb,cite,enumitem}
\usepackage{fullpage}
\usepackage{color}
\usepackage{hyperref}
\usepackage[english]{babel}
\usepackage{framed}

\newcommand{\tP}{\tilde{P}}

\newtheorem{theorem}{Theorem}
\newtheorem{lemma}{Lemma}
\newtheorem{proposition}[lemma]{Proposition}
\newtheorem{corollary}[lemma]{Corollary}

\newtheorem{remark}[lemma]{Remark}
\newtheorem{conjecture}[theorem]{Conjecture}
\numberwithin{lemma}{section}

\numberwithin{equation}{section}

\newcommand{\R}{{\mathbb R}}

\usepackage[makeroom]{cancel}

\newcommand{\sgn}{\mathop{\mathrm{sgn}}}

\newcommand{\tv}{\tilde{v}}
\newcommand{\tphi}{\tilde{\phi}}

\renewcommand{\L}{{L^{NL}}} 

\begin{document}

\title{Well-posedness and dispersive decay of small data solutions for the Benjamin-Ono equation}

\author{Mihaela Ifrim}
\address{Department of Mathematics, University of California at Berkeley}
\thanks{The first author was supported by the Simons Foundation}
\email{ifrim@math.berkeley.edu}

\author{ Daniel Tataru}
\address{Department of Mathematics, University of California at Berkeley}
 \thanks{The second author was partially supported by the NSF grant DMS-1266182
as well as by the Simons Foundation}
\email{tataru@math.berkeley.edu}

\begin{abstract}
  This article represents a first step toward understanding the long
  time dynamics of solutions for the Benjamin-Ono equation.  While
  this problem is known to be both completely integrable and globally
  well-posed in $L^2$, much less seems to be known concerning its long
  time dynamics. Here, we prove that for small localized data the
  solutions have (nearly) dispersive dynamics almost globally in time.
  An additional objective  is to revisit the $L^2$ theory for the
  Benjamin-Ono equation and provide a simpler, self-contained
  approach.
\end{abstract}

\maketitle

\tableofcontents

\section{Introduction}

In this article we consider  the  Benjamin-Ono equation
\begin{equation}\label{bo}
(\partial_t + H \partial_x^2) \phi  =   \frac12 \partial_x (\phi^2), \qquad \phi(0) = \phi_0,
\end{equation}
where $\phi$ is a real valued function $\phi :\mathbf{R}\times\mathbf{R}\rightarrow \mathbf{R}$. 
$H$ denotes the Hilbert transform on the real line; we use the convention that its symbol is
\[
H(\xi) = - i \sgn \xi
\]
as in Tao \cite{tao} and opposite to Kenig-Martel \cite{km}. Thus, dispersive waves travel to the right 
and solitons to the left.

The Benjamin-Ono equation is a model for the propagation of one
dimensional internal waves (see \cite{benjamin}).  Among others, it describes the
physical phenomena of wave propagation at the interface of layers of
fluids with different densities (see Benjamin~\cite{benjamin} and
Ono~\cite{ono}). It also  belongs to a larger class of equation modeling
this type of phenomena, some of which are certainly more physically
relevant than others.

Equation \eqref{bo} is known to be completely integrable. In
particular it has an associated Lax pair, an inverse scattering 
transform and an infinite hierarchy of conservation laws.  For further information in this direction 
we refer the reader to \cite{klm} and references therein.
We list only some of these conserved energies, which hold for smooth solutions
(for example $H_x^3(\mathbb{R})$).  Integrating by parts, one sees
that this problem has conserved mass,
\[
E_0 = \int \phi^2 \, dx,
\]
momentum
\[
E_1 = \int \phi H\phi_x - \frac13 \phi^3\, dx ,
\]
as well as energy
\[
E_2 = \int \phi_x^2 - \frac34 \phi^2 H \phi_x + \frac18 \phi^4 \, dx.
\]
More generally, at each nonnegative integer $k$ we similarly have a
conserved energy $E_k$ corresponding at leading order to the $\dot
H^{\frac{k}2}$ norm of $\phi$.

This is closely related to the  Hamiltonian structure of the equation, which uses the symplectic form
\[
\omega (\psi_1,\psi_2) = \int \psi_1 \partial_x \psi_2\, dx 
\]
with associated map $J = \partial_x$. Then the Benjamin-Ono equation is generated 
by the Hamiltonian $E_1$ and symplectic form $\omega$. $E_0$ generates the group of translations.
All higher order conserved energies can be viewed in turn as Hamiltonians
for a sequence of commuting flows, which are known as the Benjamin-Ono 
hierarchy of equations. 

The Benjamin-Ono equation is a dispersive equation, i.e. the group velocity of waves depends 
on the frequency. Precisely, the dispersion relation for the linear part is given by
\[
\omega(\xi) = - \xi |\xi|, 
\]
and the group velocity for waves of frequency $\xi$ is $v = 2|\xi|$.
Here we are considering real solutions, so the positive and negative
frequencies are matched.  However, if one were to restrict the linear
Benjamin-Ono waves to either positive or negative frequencies
then we obtain a linear Schr\"odinger equation with a choice of signs.
Thus one expects that many features arising in the study of nonlinear Schr\"{o}dinger equations
will also appear in the study of Benjamin-Ono.

Last but not least, when working with Benjamin-Ono one has to take into account its quasilinear
character.  A cursory examination of the equation might lead one to the conclusion that 
it is in effect semilinear. It is only a deeper analysis see \cite{MST}\cite{KT} which reveals the fact that the 
derivative in the nonlinearity is strong enough to insure that the nonlinearity is non-perturbative,
and that only countinuous dependence on the initial data may hold, even at high regularity.

Considering local and global well-posedness results in Sobolev spaces $H^s$,
 a natural threshold is given by the fact that the Benjamin-Ono equation has a scale invariance, 
\begin{equation}
\label{si}
\phi(t,x)  \to \lambda \phi(\lambda^2 t,\lambda x),
\end{equation}
and the scale invariant Sobolev space associated to this scaling is $\dot H^{-\frac12}$.

There have been many developments in the well-posedness theory for the
Benjamin-Ono equations, see: \cite{bp, ik, KT, KK, tao, MST, ponce,
  iorio, saut}.  Well-posedness in weighted Sobolev spaces was considered in \cite{fp} and \cite{fpl},
while soliton stability was studied in \cite{km, gusta}.
These is also  closely related  work on an extended class
of equations, called generalized Benjamin-Ono equations, for which we refer
the reader to \cite{herr}, \cite{hikk} and references therein. More extensive discussion of Benjamin-Ono and related fluid 
 models can be found in the survey papers \cite{abfs} and \cite{klein-saut}.  

Presently, for the Cauchy problem at low regularity,
the existence and uniqueness result at the level of $H^s(\mathbf{R})$
data is now known for the Sobolev index $s\geq 0$.  Well-posedness
in the  range $-\frac12 \leq s < 0$ appears to be an open question.
We now review some of the key thresholds in this analysis.

The $H^3$ well posedness result was obtained by Saut in
\cite{saut}, using energy estimates. For convenience we use his result as a starting point for
our work, which is why we recall it here:

\begin{theorem} 
The Benjamin-Ono equation is globally  well-posed in $H^3$.
\end{theorem}

The $H^1$ threshold is another important one, and it was reached by Tao~\cite{tao}; his 
article is highly relevant to the present work, and it is where the idea
of renormalization is first used in the study of Benjamin-Ono equation.

The $L^2$ threshold was first reached by Ionescu and Kenig \cite{ik},
essentially by implementing Tao's renormalization argument in the
context of a much more involved and more delicate functional setting,
inspired in part from the work of the second author \cite{mp} and of
Tao \cite{tao} on wave maps.  This is imposed by the fact that the
derivative in the nonlinearity is borderline from the perspective of
bilinear estimates, i.e. there is no room for high frequency losses.
An attempt to simplify the $L^2$ theory was later made by
Molinet-Pilod~\cite{mp}; however, their approach still involves a
rather complicated functional structure, involving not only $X^{s,b}$
spaces but additional weighted mixed norms in frequency.

Our first goal here is to revisit the $L^2$ theory for the Benjamin-Ono equation,
and (re)prove the following theorem:

\begin{theorem}\label{thm:lwp}  
The Benjamin-Ono equation is globally well-posed in $L^2$.
\end{theorem}

Since the $L^2$ norm of the solutions is conserved, this is in effect a local in time result,
trivially propagated in time by the conservation of mass. In particular it says little 
about the long time properties of the flow, which will be our primary target here.

Given the quasilinear nature of the Benjamin-Ono equation, here it is important 
to specify the meaning of well-posedness. This is summarized in the following properties:

\begin{description}
\item [(i) Existence of regular solutions] For each initial data $\phi_0 \in H^3$ there exists a unique 
global solution $\phi \in C(\R;H^3)$.

\item[(ii) Existence and uniqueness of rough solutions] For each initial data $\phi_0 \in L^2$ 
there exists a solution $\phi \in  C(\R;L^2)$, which is the unique limit of regular solutions.

\item[(iii) Continuous dependence ] The data to solution map $\phi_0 \to \phi$ is continuous 
from $L^2$ into $C(L^2)$, locally in time.

\item [(iv) Higher regularity] The data to solution map $\phi_0 \to \phi$ is  continuous 
from $H^s$ into $C(H^s)$, locally in time, for each $s > 0$.

\item[(v) Weak Lipschitz dependence] The flow map for $L^2$ solutions is locally Lipschitz
in the $ H^{-\frac12}$ topology. 
\end{description}

The weak Lipschitz dependence part appears to be a new result, even though certain estimates
for differences of solutions are part of the prior proofs in \cite{ik} and \cite{mp}.

Our approach to this result is based on the idea of normal forms,
introduced by Shatah \cite{shatah} \cite{hikk}in the dispersive realm
in the context of studying the long time behavior of dispersive pde's.
Here we turn it around and consider it in the context of studying
local well-posedness.  In doing this, the chief difficulty we face is
that the standard normal form method does not readily apply for
quasilinear equations.

One very robust adaptation of the normal form method to quasilinear
equations, called ``quasilinear modified energy method'' was
introduced earlier by the authors and collaborators in \cite{BH}, and
then further developed in the water wave context first in \cite{hit}
and later in \cite{its, itg, itgr, itc}. There the idea is to modify the energies, rather
than apply a normal form transform to the equations; this method  is then
successfully used in the study of long time behavior of solutions.
Alazard and Delort~\cite{ad, ad1} have also developed another way of constructing 
the same type of almost conserved energies by using a partial normal form transformation
to symmetrize the equation, effectively diagonalizing the leading part of the energy.

The present paper provides a different quasilinear adaptation of the normal form method.
Here we do transform the equation, but not with a direct quadratic normal form correction
(which would not work). Instead we split the quadratic nonlinearity in two parts,
a milder part and a paradifferential part\footnote{This splitting is of course not a new idea,
and it has been used for some time in the study of quasilinear problems}.
Then we construct our normal form correction in two steps: first a direct quadratic correction
for the milder part, and then a renormalization type correction for the paradifferential part.
For the second step we use a paradifferential version of Tao's renormalization argument \cite{tao}.

Compared with the prior proofs of $L^2$ well-posedness in \cite{ik}
and \cite{mp}, our functional setting is extremely simple, using only
Strichartz norms and bilinear $L^2$ bounds.  Furthermore, the bilinear
$L^2$ estimates are proved in full strength but used only in a very
mild way, in order to remove certain logarithmic divergences which
would otherwise arise.  The (minor) price to pay is that the argument
is now phrased as a bootstrap argument, same as in \cite{tao}.
However this is quite natural in a quasilinear context.

One additional natural goal in this problem is the enhanced uniqueness
question, namely to provide relaxed conditions which must be imposed
on an arbitrary $L^2$ solution in order to compel it to agree with the
$L^2$ solution provided in the theorem.  This problem has received
substantial attention in the literature but is beyond the scope of the
present paper.  Instead we refer the reader to the most up to date
results in \cite{mp}.

We now arrive at the primary goal of this paper.  The question we
consider concerns the long time behavior of Benjamin-Ono solutions
with small localized data.  Precisely, we are asking what is the
optimal time-scale up to which the solutions have linear dispersive
decay.  Our main result is likely optimal, and asserts that this holds
almost globally in time:

\begin{theorem}
Assume that the initial data $\phi_0$ for \eqref{bo} satisfies
\begin{equation}\label{data}
\| \phi_0\|_{L^2} + \| x \phi_0\|_{L^2}  \leq \epsilon \ll 1.
\end{equation} 
Then the solution $\phi$ satisfies the dispersive decay bounds
\begin{equation} \label{point}
|\phi(t,x)| +\vert H\phi (t,x)\vert  \lesssim \epsilon |t|^{-\frac12} \langle x_- t^{-\frac12}\rangle^{-\frac12}
\end{equation}
up to time 
\[
|t| \lesssim T_{\epsilon}:= e^{\frac{c}{\epsilon}}, \qquad c \ll 1.
\]
\end{theorem}

The novelty in our result is that the solution exhibits dispersive
decay.  We also remark that better decay holds in the region $x <
0$. This is because of the dispersion relation, which sends all the
propagating waves to the right. 

A key ingredient of the proof of our result is a seemingly new 
conservation law for the Benjamin Ono equation, which is akin to a 
normal form associated to a corresponding linear conservation law.

 This result closely resembles the authors' recent work in \cite{nls}
(see also further references therein) on the cubic nonlinear
Schr\"odinger problem (NLS)
\begin{equation}
i u_t - u_{xx} = \pm u^3, \qquad u(0) = u_0,
\end{equation}
with the same assumptions on the initial data.  However, our result
here is only almost global, unlike the global NLS result in \cite{nls}.

To understand why the cubic NLS problem serves as a good comparison,
we first note that both the Benjamin-Ono equation and the cubic NLS
problem have $\dot H^{-\frac12}$ scaling. Further, for a restricted
frequency range of nonlinear interactions in the Benjamin-Ono
equation, away from zero frequency, a normal form transformation turns
the quadratic BO nonlinearity into a cubic NLS type problem for which
the methods of \cite{nls} apply. Thus, one might naively expect a
similar global result.  However, it appears that the Benjamin-Ono
equation exhibits more complicated long range dynamics near frequency
zero, which have yet to be completely understood.

One way to heuristically explain these differences is provided by the
the inverse scattering point of view. While the small data cubic
focusing NLS problem has no solitons, on the other hand in the
Benjamin-Ono case the problem could have solitons for arbitrarily
small localized data.  As our result can only hold in a non-soliton
regime, the interesting question then becomes what is the lowest
time-scale where solitons can emerge from small localized data. A
direct computation\footnote{This is based on the inverse scattering theory 
for the Benjamin-Ono equation, and will be described in subsequent work.} 
shows that this is indeed the almost global time
scale, thus justifying our result. 

We further observe that our result opens the way for the next natural
step, which is to understand the global in time behavior of solutions,
where in the small data case one expects a dichotomy between dispersive solutions and
dispersive solutions plus one soliton:

\begin{conjecture}[Soliton resolution]
Any global Benjamin-Ono solution which has small data as in \eqref{data} must 
either be dispersive, or it must resolve into a soliton and a dispersive part.
\end{conjecture}

\section{Definitions and review of notations }

\subsubsection*{The big O notation:} We use the notation $A \lesssim
B$ or $A = O(B)$ to denote the estimate that $|A| \leq C B$, where $C$
is a universal constant which will not depend on $\epsilon$.  If $X$
is a Banach space, we use $O_X(B)$ to denote any element in $X$ with
norm $O(B)$; explicitly we say $u =O_X(B)$ if $\Vert u\Vert _X\leq C
B$.  We use $\langle x\rangle$ to denote the quantity $\langle x
\rangle := (1 + |x|^2)^{1/2}$.

\subsubsection*{Littlewood-Paley decomposition:} One important tool in
dealing with dispersive equations is the Littlewood-Paley
decomposition.  We recall its definition and also its usefulness in
the next paragraph. We begin with the Riesz decomposition
\[
 1 = P_- + P_+,
\]
where $P_\pm$ are the Fourier projections to $\pm [0,\infty)$; from 
\[
\widehat{Hf}(\xi)=-i\sgn (\xi)\, \hat{f}(\xi),
\]
 we observe that
\begin{equation}\label{hilbert}
iH = P_{+} - P_{-}.
\end{equation}
Let $\psi$ be a bump function adapted to $[-2,2]$ and equal to $1$ on
$[-1,1]$.  We define the Littlewood-Paley operators $P_{k}$ and
$P_{\leq k} = P_{<k+1}$ for $k \geq 0$ by defining
$$ \widehat{P_{\leq k} f}(\xi) := \psi(\xi/2^k) \hat f(\xi)$$
for all $k \geq 0$, and $P_k := P_{\leq k} - P_{\leq k-1}$ (with the
convention $P_{\leq -1} = 0$).  Note that all the operators $P_k$,
$P_{\leq k}$ are bounded on all translation-invariant Banach spaces,
thanks to Minkowski's inequality.  We define $P_{>k} := P_{\geq k-1}
:= 1 - P_{\leq k}$.

For simplicity, and because $P_{\pm}$ commutes with the
Littlewood-Paley projections $P_k$, $P_{<k}$, we will introduce the
following notation $P^{\pm}_k:=P_kP_{\pm}$ , respectively
$P^{\pm}_{<k}:=P_{\pm}P_{<k}$.  In the same spirit, we introduce the
notations $\phi^{+}_k:=P^{+}_k\phi$, and $\phi^{-}_k:=P^{-}_k\phi$,
respectively.

Given the projectors $P_k$, we also introduce additional projectors $\tilde P_k$
with slightly enlarged support (say by $2^{k-4}$) and symbol equal to $1$ in the support of $P_k$.

From Plancherel's theorem  we have the bound
\begin{equation}\label{eq:planch}
 \| f \|_{H^s_x} \approx (\sum_{k=0}^\infty \| P_k f\|_{H^s_x}^2)^{1/2}
\approx (\sum_{k=0}^\infty 2^{ks} \| P_k f\|_{L^2_x}^2)^{1/2}
\end{equation}
for any $s \in \mathbb R$.

\subsubsection*{Multi-linear expressions} We shall now make use of a
convenient notation for describing multi-linear expressions of product
type, as in \cite{tao-c}.  By $L(\phi_1,\cdots,\phi_n)$ we denote a
translation invariant expression of the form
\[
L(\phi_1,\cdots,\phi_n)(x) = \int K(y) \phi_1(x+y_1) \cdots \phi_n(x+y_n) \, dy, 
\] 
 where $K \in L^1$. More generally, one can replace $Kdy$ by any bounded measure. 
  By $L_k$ we denote such multilinear expressions whose output is localized at frequency $2^k$.
 
  This $L$ notation is extremely handy for expressions such as the
  ones we encounter here; for example we can re-express the normal
  form \eqref{commutator B_k} in a simpler way as shown in
  Section~\ref{s:local}. It also behaves well with respect to reiteration,
e.g. 
\[
L(L(u,v),w) = L(u,v,w).
\]

Multilinear $L$ type expressions can easily be estimated
in terms of linear bounds for their entries. For instance we have
\[
\Vert L(u_1,u_2) \|_{L^r} \lesssim \|u_1\|_{L^{p_1}} \|u_1\|_{L^{p_2}}, \qquad \frac{1}{p_1}+ \frac{1}{p_2} = \frac{1}{r}.
\]
A slightly more involved situation arises in this article when we seek to 
use bilinear bounds in estimates for an $L$ form. There we need to account for the 
effect of uncorrelated translations, which are allowed given the integral bound 
on the kernel of $L$. To account for that we use the translation group $\{T_y\}_{y \in \R}$,
\[
(T_y u)(x) = u(x+y),
\]
and estimate, say, a trilinear form as follows:
\[
\|L(u_1,u_2,u_3) \|_{L^r} \lesssim \|u_1\|_{L^{p_1}} \sup_{y \in \R} 
\|u_2 T_y u_3\|_{L^{p_2}}, \qquad \frac{1}{p_1}+ \frac{1}{p_2} = \frac{1}{r} .
\]
On occasion we will write this in a shorter form
\[
\|L(u_1,u_2,u_3) \|_{L^r} \lesssim \|u_1\|_{L^{p_1}}   \|L(u_2,u_3)\|_{L^{p_2}}.
\]

  To prove the boundedness in $L^2$ of the normal form transformation,
  we will use the following proposition from Tao \cite{tao-c}; for
  completeness we recall it below:

\begin{lemma}[Leibnitz rule for $P_k$]\label{commutator}  We have the commutator identity
\begin{equation}
\label{commute}
\left[ P_k\, ,\, f\right]  g = L(\partial_x f, 2^{-k} g).
\end{equation}
\end{lemma}
When classifying cubic terms (and not only) obtained after
implementing a normal form transformation, we observe that having a
commutator structure is a desired feature.  In particular Lemma
~\ref{commutator} tells us that when one of the entry (call it $g$) has
frequency $\sim 2^k$ and the other entry (call it $f$) has frequency
$\lesssim 2^k$, then $P_k(fg) - f P_k g$ effectively shifts a
derivative from the high-frequency function $g$ to the low-frequency
function $f$.  This shift will generally ensure that all such
commutator terms will be easily estimated.

\subsubsection*{Frequency envelopes.} Before stating one of the main
theorems of this paper, we revisit the \emph{frequency envelope}
notion; it will turn out to be very useful, and also an elegant tool
later in the proof of the local well-posedness result, both in the
proof of the a-priori bounds for solutions for the Cauchy problem
\eqref{bo} with data in $L^2$, which we state in
Section~\ref{s:local}, and in the proof of the bounds for the
linearized equation, in the following section.

Following Tao's paper \cite{tao} we say that a sequence $c_{k}\in l^2$
is an $L^2$ frequency envelope for $\phi \in L^2$ if
\begin{itemize}
\item[i)] $\sum_{k=0}^{\infty}c_k^2 \lesssim 1$;\\
\item[ii)] it is slowly varying, $c_j /c_k \leq 2^{\delta \vert j-k\vert}$, with $\delta$ a very small universal constant;\\
\item[iii)] it bounds the dyadic norms of $\phi$, namely $\Vert P_{k}\phi \Vert_{L^2} \leq c_k$. 
\end{itemize}
Given a frequency envelope $c_k$ we define 
\[
 c_{\leq k} = (\sum_{j \leq k} c_j^2)^\frac12, \qquad  c_{\geq k} = (\sum_{j \geq k} c_j^2)^\frac12.
\]

\begin{remark}
To avoid dealing with certain issues arising at low frequencies,
we can harmlessly make the extra assumption that $c_{0}\approx 1$.
\end{remark}

\begin{remark}
Another useful variation is to weaken the slowly varying assumption to
\[
2^{- \delta \vert j-k\vert} \leq    c_j /c_k \leq 2^{C \vert j-k\vert}, \qquad j < k,
\]
where $C$ is a fixed but possibly large constant. All the results in this paper are compatible with this choice.
This offers the extra flexibility of providing higher regularity results by the same argument.
\end{remark}

\section{The linear flow}

Here we consider  the linear Benjamin-Ono flow,
\begin{equation}\label{bo-lin}
(\partial_t + H\partial^2)\psi = 0, \qquad \psi(0) = \psi_0.
\end{equation}
Its solution $\phi(t) = e^{-t H\partial^2} \psi_0$ has conserved $L^2$ norm, and satisfies 
standard dispersive bounds:

\begin{proposition}\label{p:bodisp}
The linear Benjamin-Ono flow satisfies the dispersive bound
\begin{equation} 
 \|e^{-t H\partial^2}\|_{L^1 \to L^\infty} \lesssim t^{-\frac12}.
\end{equation}
\end{proposition}
This is a well known result. For convenience we outline the classical proof,
and then provide a second, energy estimates based proof.

\begin{proof}[First proof of Proposition~\ref{p:bodisp}]
  Applying the spatial Fourier transform and solving the corresponding
  differential equation we obtain the following solution of the linear
  Benjamin-Ono equation
\begin{equation}
\label{linear-sol}
\psi (t,x)=\int_{-\infty}^{\infty}\int_{-\infty}^{\infty} e^{-i\vert \xi\vert \xi t+i\xi (x-y)}\psi_{0}(y)\, dyd\xi.
\end{equation}
We change coordinates $\xi \rightarrow t^{-\frac{1}{2}}\eta$ and 
rewrite \eqref{linear-sol} as
\[
\psi (t,x)=t^{-\frac{1}{2}}\int_{-\infty}^{\infty}\int_{-\infty}^{\infty} e^{-i\vert \eta\vert \eta +i\eta t^{-\frac{1}{2}} (x-y)}\psi_{0}(y)\, dyd\eta ,
\]
which can be further seen as a convolution 
\[
\psi (t,x)=t^{-\frac{1}{2}} A(t^{-\frac{1}{2}}x)\ast \psi_{0}(x),
\]
where $A(x)$ is an oscillatory integral
\[
A(x):=\int_{-\infty}^{\infty}e^{-i\vert \eta \vert \eta +i\eta x}\, d\eta .
\]
It remains to show that $A$ is bounded, which follows by a standard stationary phase argument, with a minor 
complication arising from the fact that the phase is not $C^2$ at $\eta = 0$.
\end{proof}

The second proof will also give us a good starting point in our study of the dispersive properties for the
nonlinear equation. This is based on using the operator 
\[
L = x- 2 t H \partial_x ,
\]
which is the push forward of $x$ along the linear flow,
\[
L(t) = e^{-t H\partial^2} x e^{t H\partial^2} ,
\]
and thus commutes with the linear operator,
\[
[ L, \partial_t + H\partial^2] = 0.
\]
In particular this shows that for solutions $\psi$ to the homogeneous equation, the 
quantity $\|L\psi \|_{L^2}^2$ is also a conserved quantity.

\begin{proof}[Second proof of Proposition~\ref{p:bodisp}]

We rewrite the dispersive estimate  in the form
\[
\| e^{-t H \partial^2} \delta_0 \|_{L^\infty} \lesssim t^{-\frac12}.
\]
We approximate $\delta_0$ with standard bump functions $\alpha_\epsilon(x) = \epsilon^{-1} \alpha(x/\epsilon)$,
where $\alpha$ is a $C_0^\infty$ function with integral one.
It suffices to show the uniform bound
\begin{equation}\label{phie}
\| e^{-t H \partial^2} \alpha_\epsilon \|_{L^\infty} \lesssim t^{-\frac12}.
\end{equation}

The functions $\alpha_\epsilon$ satisfy the $L^2$ bound 
\[
\|\alpha_\epsilon\|_{L^2} \lesssim \epsilon^{-\frac12}, \qquad \|x \alpha_\epsilon\|_{L^2} \lesssim \epsilon^{\frac12}.
\]
By energy estimates, this implies that
\[
\|e^{-t H \partial^2}\alpha_\epsilon\|_{L^2} \lesssim \epsilon^{-\frac12}, \qquad \|L e^{-t H \partial^2} \alpha_\epsilon\|_{L^2} 
\lesssim \epsilon^{\frac12}.
\]
Then the bound \eqref{phie} is a consequence of the following

\begin{lemma}
The following pointwise bound holds:
\begin{equation}
\label{prima}
\|\psi\|_{L^\infty}+\|H\psi\|_{L^\infty} \lesssim t^{-\frac12}  \|\psi\|_{L^2}^\frac12 \| L\psi\|_{L^2}^\frac12 .
\end{equation}
\end{lemma}
We remark that the operator $L$ is elliptic in the region $x < 0$, therefore
a better pointwise bound is expected there. Indeed, we have the estimate
\begin{equation}
\label{prima+}
|\psi(t,x)| + |H \psi(t,x)|  \leq t^{-\frac12} (1+ |x_-| t^{-\frac12})^{-\frac14} \|\psi\|_{L^2}^\frac12 \| L\psi\|_{L^2}^\frac12 ,
\end{equation}
where $x_-$ stands for the negative part of $x$. To avoid repetition we do not prove this here,
but it does follow from the analysis in the last section of the paper.

\begin{proof}
Denote 
\[
c =  \int_\R \psi\,  dx .
\]
We first observe that we have
\begin{equation} \label{c-est}
c^2 \lesssim \| \psi\|_{L^2} \|L\psi\|_{L^2}.
\end{equation}
All three quantities are constant along the linear Benjamin-Ono flow, so it suffices to verify this 
at $t = 0$. But there this inequality becomes 
\[
c^2 \lesssim \| \psi\|_{L^2} \| x \psi\|_{L^2},
\]
which is straightforward using H\"older's inequality on each dyadic spatial region.

Next we establish the uniform $t^{-\frac12}$ pointwise bound. We rescale to $t = 1$.
Denote  $u = P^+ \psi$, so that $\psi = 2 \Re u$ and $H\psi = 2 \Im u$.  Hence it suffices to 
obtain the pointwise bound for $u$.

We begin with the relation
\[
( x + 2i \partial) u = P^+ L \psi + c,
\]
where the $c$ term arises from the commutator of $P^+$ and $x$. 
We rewrite this as 
\[
\partial_x ( ue^{\frac{ix^2}{4}}) = \frac{1}{2i} e^{\frac{ix^2}{4} } ( P^+ L \psi + c).
\]
Let $F$ be a bounded antiderivative for $\frac{1}{2i}  e^{\frac{ix^2}{4}} $. 
Then we introduce the auxiliary function 
\[
v =  ue^{\frac{ix^2}{4t}} -  cF,
\]
which satisfies
\[
 \partial_x v =  \frac{1}{2i} e^{\frac{ix^2}{4}} ( P^+ L \psi ).
\]
In view of the previous bound \eqref{c-est} for $c$, it remains to show that
\begin{equation}\label{v-est}
\| v\|_{L^\infty}^2  \lesssim c^2 +\| v_x\|_{L^2} \| v+cF\|_{L^2} .
\end{equation}

On each interval $I$ of length $R$ we have by H\"older's inequality
\[
\| v\|_{L^\infty(I)} \lesssim R^{\frac12} \| v_x\|_{L^2(I)} + R^{-\frac12} \| v\|_{L^2(I)}.
\]
Thus we obtain
\[
\| v\|_{L^\infty}^2 \lesssim R \|v_x\|_{L^2}^2 + R^{-1} ( \|v+cF\|_{L^2}^2 + c^2 R)   
= c^2 + R \|v_x\|_{L^2}^2 + R^{-1}  \|v+cF\|_{L^2}^2 ,
\]
and \eqref{v-est} follows by optimizing the value for $R$.

 \end{proof}
\end{proof}

One standard consequence of the dispersive estimates is the  Strichartz inequality, 
which applies to solutions to the inhomogeneous linear Benjamin-Ono equation.
\begin{equation}\label{bo-lin-inhom}
(\partial_t + H\partial^2)\psi = f, \qquad \psi(0) = \psi_0.
\end{equation}

We define the Strichartz space $S$ associated to the $L^2$ flow by
\[
S = L^\infty_t L^2_x \cap L^4_t L^\infty_x,
\]
as well as its dual
\[
S' = L^1_t L^2_x + L^{\frac43} _t L^1_x .
\]
We will also use the notation 
\[
S^{s} = \langle D \rangle^{-s} S
\]
to denote the similar spaces associated to the flow in $H^s$.

The Strichartz estimates in the $L^2$ setting are summarized in the following

\begin{lemma}
Assume that $\psi$ solves \eqref{bo-lin-inhom} in $[0,T] \times \R$. Then 
the following estimate holds.
\begin{equation}
\label{strichartz}
\| \psi\|_S \lesssim \|\psi_0 \|_{L^2} + \|f\|_{S'} .
\end{equation}
\end{lemma}

We remark that these Strichartz estimates can also be viewed as a consequence \footnote{ Exept for the $L^4_tL_x^{\infty}$ bound, as the Hilbert transform is not bounded in $L^{\infty}$.}
of the similar estimates for the linear Schr\"odinger equation. This is because the two flows
agree when restricted to functions with frequency localization in $\R^+$.

We also remark that we have the following Besov version of the 
estimates, 
\begin{equation}
\label{strichartzB}
\| \psi\|_{\ell^2 S} \lesssim \|\psi_0 \|_{L^2} + \|f\|_{\ell^2 S'},
\end{equation}
 where 
\[
\| \psi \|_{\ell^2 S}^2 = \sum_k \| \psi _k \|_{ S}^2, \qquad  \| \psi \|_{\ell^2 S'}^2 = \sum_k \| \psi _k \|_{ S'}^2 .
\]

The last property of the linear Benjamin-Ono equation we will use here is the 
bilinear $L^2$ estimate, which is as follows:

\begin{lemma}
\label{l:bi}
Let $\psi^1$, $\psi^2$ be two solutions to the inhomogeneous Schr\"odinger equation with data 
$\psi^1_0$, $\psi^2_0$ and inhomogeneous terms $f^1$ and $f^2$. Assume 
that the sets 
\[
E_i = \{ |\xi|, \xi \in \text{supp } \hat \psi^i \}
\]
are disjoint. Then we have 
\begin{equation}
\label{bi-di}
\| \psi^1 \psi^2\|_{L^2} \lesssim \frac{1}{\text{dist}(E_1,E_2)} 
( \|\psi_0^1 \|_{L^2} + \|f^1\|_{S'}) ( \|\psi_0^2 \|_{L^2} + \|f^2\|_{S'}).
\end{equation}
\end{lemma}

These bounds also follow from the similar bounds for the Schr\"odinger equation, where only the separation 
of the supports of the Fourier transforms is required. They can be obtained in a standard manner  from the similar bound for products of solutions to the homogenous equation, for which we reffer the reader to \cite{tao-m}.

One corollary of this applies in the case when we look at the product
of two solutions which are supported in different dyadic regions:
\begin{corollary}\label{c:bi-jk}
Assume that $\psi^1$ and $\psi^2$ as above are supported in dyadic regions $\vert \xi\vert \approx 2^j$ and $\vert \xi\vert \approx 2^k$, $\vert j-k\vert >2$, then
\begin{equation}
\label{bi}
\| \psi^1 \psi^2\|_{L^2} \lesssim 2^{-\frac{\max\left\lbrace j,k \right\rbrace  }{2}}
( \|\psi_0^1 \|_{L^2} + \|f^1\|_{S'}) ( \|\psi_0^2 \|_{L^2} + \|f^2\|_{S'}).
\end{equation}
\end{corollary}

Another useful  case is when we look at the product
of two solutions which are supported in the same  dyadic region, but with frequency separation:
\begin{corollary}\label{c:bi-kk}
Assume that $\psi^1$ and $\psi^2$ as above are supported in the dyadic region 
$\vert \xi\vert \approx 2^k$, but have $O(2^k)$ frequency separation between their supports.
Then 
\begin{equation}
\label{bi-kk}
\| \psi^1 \psi^2\|_{L^2} \lesssim 2^{-\frac{k}{2}}
( \|\psi_0^1 \|_{L^2} + \|f^1\|_{S'}) ( \|\psi_0^2 \|_{L^2} + \|f^2\|_{S'}).
\end{equation}
\end{corollary}

\section{Normal form analysis and a-priori bounds} 

In this section we establish apriori $L^2$ bounds for regular
($H^3_x$) solutions for the Cauchy problem \eqref{bo}.  First, we
observe from the scale invariance \eqref{si} of the equation
\eqref{bo} that it suffices to work with solutions for which the $L^2$
norm is small, in which case it is natural to consider these solutions on the time interval $[-1,1]$ (i.e., we set $T := 1$).

 Precisely we may assume that the initial satisfies
 \begin{equation}
 \label{small}
 \Vert \phi (0)\Vert_{L^2_x}\leq \epsilon .
 \end{equation}
 Then our main apriori estimate is as follows:
 
\begin{theorem}\label{apriori} Let $\phi$ be an $H^3_x$ solution to
  \eqref{bo} with small initial data as in \eqref{small}. Let $\left\lbrace
    c_k\right\rbrace _{k=0}^{\infty}\in l^2 $ so that $\epsilon c_k$
  is a frequency envelope for the initial $\phi(0)$ in $L^2$.
 Then we have the Strichartz bounds
\begin{equation}\label{u-small}
 \Vert \phi_k \Vert_{S^{0}([-1,1] \times \mathbf{R})} \lesssim \epsilon c_k,
\end{equation}
as well as the bilinear bounds
 \begin{equation}
 \label{bilinear-small}
  \Vert \phi_j \cdot \phi_k\Vert_{L^2}\lesssim 2^{-\frac{\max \left\lbrace j,k\right\rbrace}{2} }\epsilon^2 c_k\, c_j,  \qquad j\neq k.
   \end{equation}
   \end{theorem}
   Here, the implicit constants do not depend on the $H^3_x$ norm of
   the initial data $\phi(0)$, but they will depend on $\Vert
   \phi(0)\Vert_{L^2}$.  A standard iteration method will not work,
   because the linear part of the Benjamin-Ono equation does not have
   enough smoothing to compensate for the derivative in the
   nonlinearity.  To resolve this difficulty we use ideas related to
   the normal form method, first introduced by Shatah in \cite{shatah}
   in the context of dispersive PDEs. The main principle in the
   normal form method is to apply a quadratic correction to the
   unknown in order to replace a nonresonant quadratic nonlinearity by
   a milder cubic nonlinearity. Unfortunately this method does not
   apply directly here, because some terms in the quadratic correction
   are unbounded, and so are some of the cubic terms generated by the
   correction. To bypass this issue, here we develop a more favorable
   implementation of normal form analysis. This is carried out in two
   steps:
   \begin{itemize}
   \item  a partial normal form transformation which is bounded and removes some of the quadratic nonlinearity
  \item a conjugation via a suitable exponential (also called gauge transform, \cite{tao}) which removes in a bounded way the remaining part of the quadratic nonlinearity.
   \end{itemize}
 This will transform the Benjamin-Ono equation \eqref{bo} into an equation where the the quadratic terms have been removed  and replaced by cubic perturbative terms.

\subsection{The quadratic normal form analysis} In this subsection we formally derive the
normal form transformation for the Benjamin-Ono equation, \eqref{bo}.
Even though we will not make use of it directly we will still use portions 
of it to remove certain ranges of frequency interactions from
the quadratic nonlinearity. 

Before going further, we emphasizes that by a \emph{normal form} we
refer to any type of transformation which will remove nonresonant
quadratic terms; all such transformations are uniquely determined up
to quadratic terms.

The normal form idea goes back to Birkhoff which used it in the
context of ordinary differential equations.  Later, Shatah
\cite{shatah} was the first to implement it in the context of partial
differential equations.  In general, the fact that one can compute
such a normal form for a partial differential equation with quadratic
nonresonant interactions is not sufficient, unless the transformation
is invertible, and, as seen in other works, in addition, good energy
estimates are required. In the context of quasilinear equations one
almost never expects the normal form transformation to be bounded, and
new ideas are needed. In the Benjamin-Ono setting such ideas were
first introduced by Tao \cite{tao} whose renormalization is a partial normal
form transformation in disguise. More recently, other ideas have been
introduced in the quasilinear context by Wu \cite{wu}, Hunter-Ifrim
\cite{hi}, Hunter-Ifrim-Tataru \cite{BH}, Alazard-Delort \cite{ad, ad1} and Hunter-Ifrim-Tataru
\cite{hit}.

In particular, for the Benjamin-Ono equation we seek a quadratic transformation
\[
\tilde{\phi} =\phi +B(\phi, \phi),
\] 
so that the new variable $\tphi$ solves an equation with a cubic nonlinearity, 
\[
(\partial_t +H\partial_x^2)\tphi=Q(\phi, \phi, \phi),
\]
where $B$ and $Q$ are translation invariant bilinear, respectively trilinear forms. 

A direct computation yields an explicit formal spatial expression of the normal form transformation:
\begin{equation}
\label{bo nft}
\tilde{\phi}=\phi -\frac{1}{4}H\phi \cdot \partial^{-1}_x\phi  -\frac{1}{4}H\left( \phi \cdot \partial^{-1}_x\phi \right) .
\end{equation}
Note that at low frequencies \eqref{bo nft} is not invertible, which
tends to be a problem if one wants to apply the normal form
transformation directly.

\subsection{A modified normal form analysis}
\label{s:local}

We begin by writing the Benjamin-Ono equation \eqref{bo} in a
paradifferential form, i.e., we localize ourselves at a frequency $2^k$,
and then project the equation either onto negative or positive
frequencies:
\[
(\partial_t \mp  i\partial^{2}_{x} )\phi_k^{\pm}=P_k^{\pm} (\phi \cdot \phi_{x}).
\]
Since $\phi$ is real, $\phi^- $ is the complex conjugate of $\phi^+$ so it suffices to work with the latter. 

Thus, the Benjamin-Ono equation for the positive frequency Littlewood-Paley components $\phi^{+}_k$ is
\begin{equation}
\label{eq-loc}
\begin{aligned}
&\left( i\partial_{t}  + \partial^{2}_{x} \right)  \phi^+_{k}=iP_k ^{+}(\phi \cdot \phi_{x}).
\end{aligned}
\end{equation}

Heuristically, the worst term in $P_k^+(\phi \cdot \phi_{x})$ occurs
when $ \phi _x$ is at high frequency and $\phi$ is at low
frequency. We can approximate $P_k^+ (\phi \cdot \phi_x)$, by its
leading paradifferential component $\phi_{<k} \cdot \partial_x
\phi_{k}^+$; the remaining part of the nonlinearity will be harmless.
More explicitly we can eliminate it by means of a bounded normal form
transformation.

We will extract out the main term
$i\phi_{<k}\cdot \partial_{x}\phi_k^+$ from the right hand side
nonlinearity and move it to the left, obtaining
\begin{equation}
\label{eq-lin-k}
\left( i\partial_t+\partial_x^2 -i\phi_{<k}\cdot \partial_{x}\right)  \phi^+_k =i P_{k}^{+}\left( \phi_{\geq k} \cdot \phi _{x}\right) +i\left[ P_{k}^{+}\, , \, \phi_{<k} \right]\phi_x.
\end{equation}
For reasons which will become apparent later on when we do the
exponential conjugation, it is convenient to add an additional lower
order term on the left hand side (and thus also on the right).
Denoting by $A^{k,+}_{BO} $ the operator
\begin{equation}
\label{op Abo}
A^{k, +}_{BO}:= i\partial_t +\partial^2 _x-i\phi_{<k}\cdot \partial_{x}  +\frac{1}{2}  \left( H+i \right)\partial_x\phi _{<k}
\end{equation} 
we rewrite the equation \eqref{eq-lin-k} in the form
\begin{equation}
\label{eq-lin-k+}
A^{k,+}_{BO} \ \phi^+_k =i P_{k}^{+}\left( \phi_{\geq k} \cdot \phi _{x}\right) +i\left[ P_{k}^{+}\, , \, \phi_{<k} \right]\phi_x +\frac{1}{2}  \left( H+i \right)\partial_x\phi _{<k} \cdot \phi^+_k.
\end{equation}
Note the key property that the operator $A^{k,+}_{BO} $ is symmetric,
which in particular tells us that the $L^2$ norm is conserved in the
corresponding linear evolution.

 The case $k=0$ is mildly different  in this discussion. There we need no paradifferential component,
and also we want to avoid the operator $P_0^+$ which does not have a smooth symbol.
Thus we will work with the equation
 \begin{equation}
\label{eq-lin-0}
(\partial_t + H \partial_x^2)\phi_0 = P_0( \phi_{0} \phi_x) + P_0( \phi_{>0} \phi_x),
\end{equation}
where the first term on the right is purely a low frequency term and will play only a perturbative role.

The next step is to eliminate the terms on the right hand side of
\eqref{eq-lin-k+} using a normal form transformation
 \begin{equation}
\label{partial nft}
\begin{aligned}
 \tilde{\phi}_k^+:= \phi^+_k +B_{k}(\phi, \phi).
\end{aligned}
\end{equation} 
Such a transformation is easily computed and formally is given by the expression
\begin{equation}
\label{bilinear nft}
\begin{aligned}
B_{k}(\phi, \phi) =&\frac{1}{2}HP_{k}^{+}
\phi \cdot \partial_{x}^{-1}P_{<k}\phi-\frac{1}{4}P_{k} ^{+}\left( H\phi\cdot \partial_{x}^{-1}\phi\right)
 -\frac{1}{4}P_{k}^{+}H\left( \phi \cdot \partial^{-1}_x\phi\right).
\end{aligned}   
\end{equation}
One can view this as a subset of the normal form transformation
computed for the full equation, see \eqref{bo nft}. Unfortunately, as
written, the terms in this expression are not well defined because
$\partial^{-1}_x \phi$ is only defined modulo constants. To avoid this
problem we separate the low-high interactions which yield a well
defined commutator, and we rewrite $B_{k}(\phi, \phi)$ in a better
fashion as
\begin{equation}
\label{commutator B_k}
B_{k}(\phi, \phi)=-\frac{1}{2}\left[ P^+_kH\, , \, \partial^{-1}_x \phi_{<k} \right]\phi -\frac{1}{4}P^+_k \left( H\phi \cdot \partial_x^{-1}\phi_{\geq k}\right)  
-\frac{1}{4}P^+_k H \left(\phi \cdot \partial_x^{-1}\phi_{\geq k}\right).
\end{equation}

In the case $k=0$ we will keep the first term on the right and apply a quadratic correction to remove the second.
This yields
\begin{equation}
\label{commutator B_0}
B_{0}(\phi, \phi)= -\frac{1}{4}P^+_0 \left[ H\phi \cdot \partial_x^{-1}\phi_{\geq 1}\right]  
-\frac{1}{4}P^+_0 H \left[\phi \cdot \partial_x^{-1}\phi_{\geq 1}\right].
\end{equation}

\begin{remark} The normal form transformation associated to
  \eqref{eq-loc} is the normal form derived in \eqref{bo nft}, but
  with the additional $P_k^+$ applied to it.  Thus, the second and the
  third term in \eqref{bilinear nft} are the projection $P^+_k$ of
  \eqref{bo nft}, which , in particular, implies that the linear
  Schr\"odinger operator $i\partial_t+\partial_x^2$ applied to these
  two terms will eliminate entirely the nonlinearity $P_k^+ (\phi
  \cdot\phi_{x})$.  The first term in \eqref{bilinear nft} introduces
  the paradifferential corrections moved to the left of
  \eqref{eq-lin-k+}, and also has the property that it removes the
  unbounded part in the second and third term.
\end{remark}

Replying $\phi^+_k$ with $\tilde{\phi}^+_k $ removes all the quadratic
terms on the right and leaves us with an equation of the form
 \begin{equation}
 \label{eq after 1nft}
  A^{k,+}_{BO} \, \tilde{\phi}^{+}_k =Q^{3}_k(\phi, \phi, \phi ) ,
 \end{equation}
 where $Q^{3}_k(\phi, \phi, \phi)$ contains only cubic  terms in $\phi$. We will examine $Q^3_k(\phi, \phi, \phi)$ in greater detail later in Lemma~\ref{l:q3}, where
 its full expression is given.
 
  The case $k=0$ is again special. Here the first normal form transformation does not eliminate the low-low frequency interactions, and our intermediate  equation has the form
  \begin{equation}
 \label{eq after 1nft-0}
 (i\partial_t +\partial^2_x) \, \tilde{\phi}^{+}_0 =Q^2_0 (\phi, \phi) +Q^{3}_0(\phi, \phi, \phi ),
 \end{equation}
where $Q^2_0$ contains all the low-low frequency interactions
\[
Q^2_0 (\phi, \phi):=P_0^+ \left(  \phi_{0}\cdot \phi_x\right). 
\]

The second stage in our normal form analysis is to perform a second bounded normal form transformation that will remove the paradifferential terms in
the left hand side of \eqref{eq after 1nft}; this will be a renormalization, following the idea introduced by Tao (\cite{tao}). To achieve this we  introduce and initialize the spatial primitive $\Phi(t, x)$ of $\phi (t,x)$, exactly as in Tao \cite{tao}.
 It turns out that $\Phi (t,x)$ is necessarily a real valued function that  solves the equation
\begin{equation}
\label{Phi}
\Phi_{t}+H\Phi_{xx}=\Phi_{x}^2,
\end{equation}
which holds globally in time and space. Here, the initial condition imposed is $\Phi(0,0)=0$. Thus,  
\begin{equation}
\label{antiderivative}
\Phi_x (t,x)= \frac{1}{2}\phi(t,x).
\end{equation}
The idea in \cite{tao} was that in order to get bounds on $\phi$ it suffices to obtain appropriate bounds 
 on $\Phi(t,x)$ which are one higher degree of  reqularity as \eqref{antiderivative} suggests. 
Here we instead use $\Phi$ merely in an auxiliary role, in order to define the second normal form 
transformation. This is
 \begin{equation}
 \label{conjugation}
 \displaystyle{\psi_k^+:=\tilde{\phi}_k^{+} \cdot    e^{-i\Phi_{<k}}}.
 \end{equation}

 The transformation \eqref{conjugation} is akin to a Cole-Hopf
 transformation, and expanding it up to quadratic terms, one observes
 that the expression obtained works as a normal form transformation,
 i.e., it removes the paradifferential quadratic terms.  The
 difference is that the exponential will be a bounded transformation,
 whereas the corresponding quadratic normal form is not.  One also
 sees the difference reflected at the level of cubic or higher order
 terms obtained after implementing these transformation (obviously
 they will differ).

  By applying this \emph{Cole-Hopf type} transformation, we  rewrite  the equation \eqref{eq after 1nft}
as a     a nonlinear Schr\"odinger equation for our final normal form variable $\psi_k$,   with only cubic and quartic nonlinear terms:
 \begin{equation}
 \label{conjugare}
 \begin{aligned}
(i\partial_t +\partial^2_x)\, \psi_k^+ = [\tilde{Q_k}^{3}(\phi, \phi, \phi) + \tilde{Q_k}^{4}(\phi, \phi, \phi,\phi)] e^{-i\Phi_{<k}},
 \end{aligned}
 \end{equation}
 where $\tilde{Q}_k^3$ and $\tilde Q_k^4$ contain only cubic,
 respectively quartic terms; these are also computed in
 Lemma~\ref{l:q3}.
 
 The case $k=0$ is special here as well, in that no renormalization is
 needed. There we simply set $\psi_0 =\tilde{\phi}_0$, and use the
 equation \eqref{eq after 1nft-0}.

 This concludes the algebraic part of the analysis. Our next goal is
 study the analytic properties of our multilinear forms:

\begin{lemma}
\label{l:q3}
The quadratic form $B_k$ can be expressed as
\begin{equation}
B_k(\phi, \phi)=2^{-k} L_k (\phi_{<k}, \phi_k) +\sum _{j\geq k} 2^{-j}L_k (\phi_j, \phi_j)=2^{-k} L_k(\phi,\phi).
\end{equation}
The cubic and quartic expressions $Q_k^3$, $\tilde Q_k^3$ and $\tilde Q_k^4$ are translation invariant multilinear forms of  the type
\begin{equation}
\begin{aligned}
  Q_k^3(\phi,\phi,\phi) = & \ L_k(\phi,\phi, \phi)+L_k(H\phi,\phi, \phi),\\
 \tilde Q_k^3(\phi,\phi,\phi) = & \ L_k(\phi,\phi, \phi)+L_k(H\phi,\phi, \phi), 
\\
 \tilde Q_k^4(\phi,\phi,\phi,\phi) = \ &  L_k(\phi,\phi, \phi,\phi)+ L_k(H\phi,\phi, \phi,\phi),
\end{aligned}
\end{equation}
all with output at frequency $2^k$.
\end{lemma}
 \begin{proof}
   We recall that $B_k$ is given in \eqref{commutator B_k}. For the
   first term we use Lemma~\ref{commutator}. For the two remaining
   terms we split the unlocalized $\phi$ factor into $\phi_{<k}+
   \phi_{\geq k}$. The contribution of $\phi_{<k}$ is as before, while
   in the remaining bilinear term in $\phi_{\geq k}$ the frequencies
   of the two inputs must be balanced at some frequency $2^j$ where
   $j$ ranges in the region $j \geq k$. For the last expression of
   $B_k$ we simply observe that
\begin{equation}
\label{dipi}
\partial^{-1}_x\phi_{\geq k}=2^{-k}L(\phi).
\end{equation} 

Next we consider $Q^3_k$ which is obtained by a direct computation
\begin{equation}
\label{q3}
\begin{aligned}
Q^3_k(\phi, \phi, \phi)=&-\frac{1}{2}i \, \left[ P^+_kH\, , \, P_{<k}(\phi ^2)\right]\, \phi -\frac{1}{2}i \, \left[ P^+_kH\, , \, \partial_x^{-1}\phi_{<k} \right]\partial_x (\phi ^2)
-\frac{1}{4}iP^+_k\left( H\partial_x (\phi ^2)\cdot \partial_x^{-1}\phi_{\geq k}\right)\\
&-\frac{1}{4}iP^+_k\left( H\phi \cdot P_{\geq k}(\phi ^2)\right)-\frac{1}{4} iP^+_k H\left( \partial_x (\phi^2)\cdot  \partial_x^{-1}\phi_{\geq k}\right)
-\frac{1}{4}i P^+_k H\left(\phi \cdot P_{\geq k}(\phi ^2)\right)\\ 
&-iP_{<k}\phi \cdot \left\lbrace  -\frac{1}{2} \, \left[ P^+_kH\, , \, \phi_{<k}\right]\, \phi -\frac{1}{2} \, \left[ P^+_kH\, , \, \partial_x^{-1}\phi_{<k} \right]\, \phi_x
-\frac{1}{4}P^+_k\left( H\phi_x \cdot  \partial_x^{-1}\phi_{\geq k}\right)\right.\\ 
& \hspace*{2.15cm} \left. -\frac{1}{4} P^+_k  \left( H\phi \cdot  \phi_{\geq k}\right) -\frac{1}{4}P^+_KH \left( \phi_x \cdot  \partial_x^{-1}\phi_{\geq k}\right)
-\frac{1}{4}P^+_kH \left( \phi \cdot  \phi_{\geq k}\right) \right\rbrace \\
&-\frac{1}{2}\partial_x (H+i)\phi_{<k} \cdot B_{k}(\phi , \phi).
\end{aligned}
\end{equation}
We consider each term separately. For the commutator terms we use Lemma~\ref{commutator} to eliminate all the inverse derivatives. This yields a factor of $2^{-k}$ which in 
turn is used to cancel the remaining derivative in the expressions. For instance consider the second term
\[
\begin{aligned}
\left[  P^+_k H\, , \, \partial^{-1}_x  \phi_{<k} \right]\, \partial_{x}(\phi^2) &=\left[  P^+_k H\, , \, \partial^{-1}_x  \phi_{<k} \right]\,\tP_k \partial_{x}(\phi^2)\\
&=L (\phi_{<k} , 2^{-k}\tP_k\partial_{x}(\phi^2))\\
&=L (\phi_{<k} , \phi^2)\\
&=L (\phi_{<k} , \phi, \phi).
\end{aligned}
\]
The remaining terms are all similar. We consider for example the third term
\[
P^+_k \left( H\partial_x (\phi^2)\cdot \partial^{-1}_x\phi_{\geq k}\right)=  P^+_k \partial_x\left( H (\phi^2)\cdot \partial^{-1}_x\phi_{\geq k}\right)
-P^+_k \left( H (\phi^2)\cdot \phi_{\geq k}\right).
\]
The derivative in the first term yields a $2^k$ factor, and we can use \eqref{dipi}, and the second term is straightforward. 

For $\tilde{Q}^3_k$ an easy computation yields
\[
\tilde{Q}^3_k(\phi, \phi, \phi)=Q^3_k(\phi, \phi, \phi)+\frac12 \phi^+_k \cdot P_{<k}(\phi^2) -\frac14 \phi^+_k \cdot \left( P_{<k}\phi\right) ^2,
\]
and both extra terms are straightforward. 

Finally, $\tilde{Q}^4_k (\phi, \phi, \phi, \phi)$ is given by
\[
\tilde{Q}^4_k (\phi, \phi, \phi, \phi)=\frac14 B_{k}(\phi, \phi) \cdot \left\lbrace  2P_{<k}(\phi ^2) -\left( P_{<k}\phi\right)^2   \right\rbrace ,
\]
and the result follows from the one for the $B_k (\phi, \phi)$.
 
 \end{proof}

  \subsection{The bootstrap argument}
  
  We now finalize the proof of Theorem~\ref{apriori} using a standard
  continuity argument based on the $H^3_x$ global well-posedness
  theory.  Given $0 < t_0 \leq 1$ we denote by
\[
M(t_0) := \sup_{j} c_j^{-2} \,  \|  P_k \phi\|_{S^0[0,t_0]}^2 + \sup_{j\neq k\in \mathbf{N}} \sup_{y\in \R} c^{-1}_j \cdot c_k^{-1} \cdot \| \phi_j \cdot T_y \phi_k \|_{L^2\left[ 0, t_0\right] }. 
\]
Here, in the second term, the role of the condition $j\neq k$ is to
insure that $\phi_j$ and $\phi_k$ have $O(2^{\max \left\lbrace
    j,k\right\rbrace })$ separated frequency localizations.  However,
by a slight abuse of notation, we also allow bilinear expressions of
the form $ P_{k}^1\phi \cdot P_k^2\phi$, where $P_{k}^1$ and $P_{k}^2$
are both projectors at frequency $2^k$ but with at least $2^{k-4}$
separation between the \textbf{absolute values} of the frequencies in
their support.

We  also remark here on the role played by the translation operator $T_y$. This is 
needed in order for us to be able to use thee bilinear bounds in estimating multilinear 
$L$ type expressions.

 We seek to show that 
 \[
 M(1) \lesssim \epsilon^2.
 \]
  As $\phi$ is an $H^3$ solution, it is easy to see that $M(t)$
 is continuous as a function of $t$, and 
 \[
 \lim_{t\searrow 0}M(t) \lesssim \epsilon^2 .
 \]
This is because the only nonzero component of the $S$ norm in the limit $t \to 0$ 
is the energy norm, which converges to the energy norm of the data.

  Thus, by a continuity argument 
it suffices to make the bootstrap assumption 
\[
M(t_0) \leq C^2 \epsilon^2 
\]
and then show that  
\[
M(t_0) \lesssim \epsilon^2 + C^6 \epsilon^6.
\]
This suffices provided that $C$ is large enough (independent of
$\epsilon$) and $\epsilon$ is sufficiently small (depending on $C$).
From here on $t_0 \in (0,1]$ is fixed and not needed in the argument,
so we drop it from the notations.
 
Given our bootstrap assumption, we have the starting estimates
\begin{equation}
\label{boot1}
  \Vert \phi_k\Vert_{S^0}\lesssim C \epsilon c_k ,
   \end{equation}
and
\begin{equation}
  \label{boot2}
  \Vert \phi_j \cdot T_y \phi_k\Vert_{L^2}\lesssim 2^{-\frac{\max \left\lbrace j,k\right\rbrace}{2} }C^2 \epsilon^2 c_j c_k,  \qquad j\neq k, \qquad y \in \R.
\end{equation}
where in the bilinear case, as discussed above, we also allow $j=k$ provided the two localization multipliers
are at least $2^{k-4}$ separated. This separation threshold is fixed once and for all. On the other hand, when 
we prove that the bilinear estimates hold, no such sharp threshold is needed.

Our strategy will be to establish these bounds for the normal form variables $\psi_k$, and then to transfer
them to the original solution $\phi$ by inverting the normal form transformations and estimating errors.

We obtain bounds for the normal form variables $\psi_k^+$. For this we estimate the initial data for $\psi_k$ 
in $L^2$, and  then the  right hand side in the  Schrodinger equation \eqref{conjugare} for $\psi_k^+$ in 
$L^1 L^2$. For the initial data we have
\begin{lemma}
\label{l:invertibila}
Assume \eqref{small}. Then we have 
\begin{equation}
\| \psi_k^+(0)\|_{L^2} \lesssim c_k\epsilon.
\end{equation}
\end{lemma}
\begin{proof} We begin by recalling the definition of $\psi (t, x)$:
\[
\psi (t,x)=\tilde{\phi}^+_k e^{-i\Phi _{< k}}.
\] 
The $L^2_x$ norms of $\psi_k$ and $\tilde{\phi}^+_k$ are equivalent
since the conjugation with the exponential is harmless. Thus, we need
to prove that $L^2$ norm of $\tilde{\phi}^+_k$ is comparable with the
$L^2$ norm of $\phi^+_k$. The two variables are related via the
relation \eqref{partial nft}. Thus, we reduce our problem to the study
of the $L^2$ bound for the bilinear form $B_{k}(\phi , \phi)$.  From
Lemma \eqref{l:q3} we know that
\[
B_k(\phi, \phi)=2^{-k} L_k (\phi_{<k}, \phi_k) +\sum _{j\geq k} 2^{-j}L_k (\phi_j, \phi_j),
\]
so we estimate each term separately. For the first term we use the the smallness of the initial data in the $L^2$ norm, together with Bernstein's inequality, which we use for  the low frequency term 
\begin{equation*}
\Vert 2^{-k} L_k (\phi_{<k}, \phi_k)\Vert_{L^2}\lesssim  2^{-\frac{k}{2}} \cdot \epsilon  \cdot \Vert \phi(0)\Vert_{L^2}=2^{-\frac{k}{2}} \cdot \epsilon ^2 \cdot c_k.
\end{equation*}
For the second component of $B_k(\phi, \phi)$, we again use Bernstein's inequality
\[
\Vert \sum _{j\geq k} 2^{-j}L_k (\phi_j, \phi_j) \Vert_{L^2_x}\lesssim \sum_{j\geq k} 2^{-\frac{j}{2}}\cdot \epsilon  \cdot  \Vert \phi_j(0) \Vert_{L^2}\lesssim \sum_{j\geq k} 2^{-\frac{j}{2}}\cdot \epsilon ^2 \cdot c_j\lesssim  2^{-\frac{k}{2}} \cdot  c_k  \cdot \epsilon^2  .
\]

This concludes the proof.

\end{proof}

Next we consider the right hand side in the $\psi_k$ equation:
\begin{lemma}
Assume \eqref{boot1} and \eqref{boot2}. Then we have 
\begin{equation}
\label{perturbative}
\| \tilde Q_k^3 \|_{L^1 L^2} +\| \tilde Q_k^4 \|_{L^1 L^2} \lesssim C^3 \epsilon^3 c_k.
\end{equation}
\end{lemma}
A similar estimate holds for the quadratic term $Q^2_0$ which appears in the case $k=0$, 
but that is  quite straightforward.
\begin{proof}
We start  by estimating the first term in \eqref{perturbative}. For completeness we recall the expression of
 $\tilde{Q}^3_k$  from Lemma~\ref{l:q3}:
\[
\tilde{Q}^3_k(\phi, \phi, \phi)=L_{k}(\phi, \phi, \phi)+L_{k}(H\phi, \phi, \phi).
\]
Here $H$ plays no role so it suffices to discuss the first term. To
estimate the trilinear expression $L_k (\phi, \phi, \phi)$ we do a
frequency analysis. We begin by assuming that the first entry of $L_k$
is localized at frequency $2^k_1$, the second at frequency $2^{k_2}$,
and finally the third one is at frequency $2^{k_3}$. As the output is
at frequency $2^k$, there are three possible cases:
\begin{itemize}
\item If $2^k<2^{k_1}<2^{k_2}=2^{k_3}$, then we can use the bilinear Strichartz estimate for the imbalanced frequencies, and the Strichartz inequality for the remaining term  to arrive at
\[
\begin{aligned}
\Vert L_k(\phi_{k_1}, \phi_{k_2}, \phi_{k_3})\Vert_{L^{\frac43}_t L^2_x}&\lesssim \Vert L(\phi_{k_1}, \phi_{k_3})\Vert_{L^2_{t,x}}\cdot \Vert \phi_{k_2}\Vert_{L^{4}_tL^\infty_x}\\
&\lesssim 2^{-\frac{k_3}{2}}\cdot C^2\epsilon^2\cdot c_{k_1}\cdot c_{k_3}\cdot \Vert \phi_{k_2}\Vert_{L^{4}_t L^\infty_x}\\
&\lesssim 2^{-\frac{k_3}{2}}\cdot C^3\cdot \epsilon^3 \cdot  c_{k_1}c_{k_2}c_{k_3}\lesssim 2^{-\frac{k}{2}}\cdot C^3\cdot \epsilon^3 \cdot  c_{k}^3.
\end{aligned}
\]

\item If $ 2^{k_1}=2^{k_2}=2^{k_3} \approx 2^k$, then we use directly the Strichartz estimates 
\[
\Vert L_k(\phi_{k_1}, \phi_{k_2}, \phi_{k_3})\Vert_{L^2_t L^2_x}\lesssim \| \phi_{k_1}\|_{L^\infty_t L^2_x} \cdot \Vert \phi_{k_2}\Vert_{L^{4}_t L^\infty_x} \cdot \Vert \phi_{k_3}\Vert_{L^{4}_t L^\infty_x} \lesssim C^3 \epsilon^3 c_k^3.
\]
\item If $ 2^{k_1}=2^{k_2}=2^{k_3} \gg 2^k$ then the frequencies of
  the three entries must add to $O(2^k)$. Then the absolute values of
  at least two of the three frequencies must have at least a
  $2^{k_3-4}$ separation.  Thus, the bilinear Strichartz estimate
  applies, and the same estimate as in the first case follows in the
  same manner.
\end{itemize}
This concludes the bound for $\tilde Q^{3}_k$. 

Finally, the $L^1_tL^2_x$ bound for 
\[
\tilde{Q}^4_k (\phi, \phi, \phi, \phi)=\frac14 B_{k}(\phi, \phi) \cdot \left\lbrace  2P_{<k}(\phi ^2) -\left( P_{<k}\phi\right)^2   \right\rbrace ,
\]
follows from the $L^2$ bound for $B_k(\phi, \phi)$ obtained in
Lemma~\ref{l:invertibila} together with the $L^{4}_t L^\infty_x $
bounds for the remaining factors. To bound these terms we do similar
estimates as the ones in Lemma~\ref{l:invertibila}.
\end{proof}

Given the bounds in the two above lemmas we have the Strichartz estimates for $\psi_k$:
\[
\Vert \psi_k\Vert_{S^0}\lesssim \Vert \psi_k(0)\Vert _{L^2_x}+\Vert \tilde{Q}^3_k(\phi, \phi, \phi) +\tilde{Q}^4_{k}(\phi, \phi, \phi, \phi)  \Vert_{L^1_tL^2_x}
\lesssim c_k\left( \epsilon + \epsilon ^3C^3\right) .
\]
This implies the same estimate for $\tilde{\phi}^+_k$. Further we
claim that the same holds for $\phi^+_k$. For this we need to estimate
$B_{k}(\phi, \phi)$ in $S^0$. We recall that
\[
B_k(\phi, \phi) = 2^{-k} L_k(P_{<k}\phi,P_k\phi)+\sum_{j\geq k} 2^{-j} L_{k} (\phi_j, \phi_j).
\]
We now estimate
\[
\begin{aligned}
\Vert B_{k}(\phi, \phi)\Vert_{S^0}& \lesssim 2^{-k}\Vert \phi_k\Vert_{S^0}\Vert \phi_{<k}\Vert_{L^{\infty}}+\sum_{j\geq k} 2^{-j} \Vert \phi_j\Vert_{S^0}\Vert \phi_j\Vert_{L^{\infty}}\\
 &\lesssim C\epsilon^2 c_k 2^{-\frac{k}{2}}+\sum_{j\geq k}  C\epsilon^2 c_j 2^{-\frac{j}{2}}\\
 & \lesssim C\epsilon^2 c_k 2^{-\frac{k}{2}} .
 \end{aligned}
\]
Here we have used Bernstein's inequality to estimate the $L^{\infty}$ norm
in term of the mass, and the slowly varying property of the $c_k$'s
for the last series summation.  This concludes the Strichartz
component of the bootstrap argument.

For later use, we observe that the same argument as above but with without using Bernstein's 
inequality yields the bound
\begin{equation}\label{psi-err}
\|\psi_k - e^{-i \Phi_{<k}} \phi_k^+ \|_{L^2 L^\infty \cap L^4 L^2} \lesssim 2^{-k} \epsilon^2C^2 c_k
\end{equation}
as a consequence of a similar bound for $B_k$.

We now consider the bilinear estimates in our bootstrap argument. We
drop the translations from the notations, as they play no role in the
argument. Also to fix the notations, in what follows we assume that $j < k$.
The case when $j=k$ but we have frequency separation is completely similar.

We would like to start from the bilinear bounds for $\psi_k$, which
solve suitable inhomogeneous linear Schr\"odinger equations. However,
the difficulty we face is that, unlike $\tphi_k^+$, $\psi_k$ are no
longer properly localized in frequency, therefore for $j \neq k$,
$\psi_j$ and $\psi_k$ are no longer frequency separated.  To remedy
this we introduce additional truncation operators $\tilde P_j$ and
$\tilde P_k$ which still have $2^{\max\{j,k\}}$ separated supports but
whose symbols are identically $1$ in the support of $P_j$,
respectively $P_k$.  Then the bilinear $L^2$ bound in Lemma~\ref{l:bi}
yields
\[
\| \tilde P_j \psi_j \cdot \tilde P_k \psi_k\|_{L^2} \lesssim \epsilon^2 c_j c_k 2^{-\frac{\max\{j,k\}}2} 
(\epsilon^2 + C^6 \epsilon^6) .
\]
It remains to transfer this bound to $\phi_j^+ \phi_k^+$. We expand  
\[
\tilde P_j \psi_j  \tilde P_k \psi_k -  \phi_j^+ e^{-i \Phi_{<j}} \phi_k^+  e^{-i \Phi_{<k}} =
\tilde P_j \psi_j  (\tilde P_k \psi_k -  \phi_k^+ e^{-i \Phi_{<k}})+
(\tilde P_j \psi_j  -   \phi_j^+ e^{-i \Phi_{<j}})   \phi_k^+  e^{-i \Phi_{<k}} .
\]
For the first term we use the bound \eqref{psi-err} for the second factor 
combined with the Strichartz bound for the second,
\[
\|\tilde P_j \psi_j  (\tilde P_k \psi_k -  \phi_k^+ e^{-i \Phi_{<k}})\|_{L^2} \lesssim 
\| \psi_j \|_{L^\infty L^2} \|  \psi_k -  \phi_k^+ e^{-i \Phi_{<k}}\|_{L^2 L^\infty} \lesssim \epsilon^3 C^2 c_j c_k 
2^{-k},
\]
which is better than we need.  It remains to consider the second term,
where we freely drop the exponential. There the above argument no longer suffices, 
as it will only yield a $2^{-k}$ low frequency gain.

We use the commutator Lemma~\ref{commutator} to express the difference in the second term
as 
\[
\begin{split}
\tilde P_j \psi_j  -   \phi_j^+ e^{-i \Phi_{<j}}  = & \ 
(\tilde P_j-1) (\tphi_j^+ e^{  -i \Phi_{<j}}) + B_j (\phi,\phi) e^{-i \Phi_{<j}} 
\\
= & \ 
[\tilde P_j-1, e^{  -i \Phi_{<j}}] \phi_j^+ e^{  -i \Phi_{<j}}) +
(\tilde P_j-1) ( B_j (\phi,\phi)  e^{  -i \Phi_{<j}} +
+  B_j (\phi,\phi) e^{-i \Phi_{<j}} 
\\ 
= & \ 2^{-j} L(\partial_x   e^{  -i \Phi_{<j}}, \phi_j^+)  + L( B_j (\phi,\phi), e^{-i \Phi_{<j}}) 
\\
= & \  2^{-j} L( \phi_{<j}, \phi_j,  e^{  -i \Phi_{<j}}) + \sum_{l > j} 2^{-l} L(\phi_l,\phi_l, e^{-i \Phi_{<j}}) .
\end{split}
\]

Now we multiply this by $\phi_k^+$, and estimate in $L^2$ using our 
bootstrap hypothesis. For $l \neq k$ we can use a bilinear $L^2$ estimate 
combined with an $L^\infty$ bound obtained via Bernstein's inequality.
For $l = k$ we use three Strichartz bounds. The exponential is harmlessly discarded in all cases.
We obtain
\[
\| (\tilde P_j \psi_j  -   \phi_j^+ e^{-i \Phi_{<j}} ) \phi_k^{+}\|_{L^2}
\lesssim \epsilon^3 C^2 ( c_j c_k 2^{-\frac{j}2} 2^{-\frac{k}2} + \sum_{l > j}
 c_l c_k 2^{-\frac{l}2} 2^{-\frac{k}2}) = \epsilon^3 C^2  c_j c_k 2^{-\frac{j}2} 2^{-\frac{k}2}
\]
which suffices.

\section{Bounds for the linearized equation}

In this section we consider the  linearized Benjamin-Ono equation equation,
\begin{equation}\label{lin}
(\partial_t +H \partial^2_x) v  =   \partial_x  ( \phi v ).
\end{equation}
Understanding the properties of the linearized flow is critical for
any local well-posedness result. 

Unfortunately, studying the linearized problem in $L^2$ presents
considerable difficulty. One way to think about this is that $L^2$
well-posedness for the linearized equation would yield Lipschitz
dependence in $L^2$ for the solution to data map, which is known to be
false. 

Another way is to observe that by duality, $L^2$ well-posedness
implies $\dot H^{-1}$ well-posedness, and then, by interpolation,
$\dot H^s$ well-posedness  for $s \in [0,1]$. This last 
consideration shows that the weakest (and most robust)
local well-posedness result we could prove for the linearized 
equation is in $\dot H^{-\frac12}$.  

Since we are concerned with local well-posedness here, we will harmlessly replace
the homogeneous space $\dot H^{-\frac12}$ with $ H^{-\frac12}$.  Then we will prove 
the following:

\begin{theorem}
\label{t:liniarizare}
Let $\phi$ be an $H^3$  solution to the Benjamin-Ono equation 
in $[0,1]$ with small mass, as in \eqref{small}.
Then the linearized equation \eqref{lin} is well-posed in $H^{-\frac12}$
with a uniform bound 
\begin{equation}\label{S-lin}
\| v\|_{C(0,1;H^{-\frac12})} \lesssim \|v_0\|_{H^{-\frac12}}
\end{equation}
with a universal implicit constant (i.e., not depending on the $H^3$ norm of $\phi$).
\end{theorem}

We remark that as part of the proof we also show that the solutions to the linearized equation
satisfy appropriate Strichartz and bilinear $L^2$ bounds expressed in terms of the frequency envelope
of the initial data.

The rest of the section is devoted to the proof of the theorem. 
We begin by considering more regular solutions:

\begin{lemma}
\label{l:dippi}
Assume that $\phi$ is an $H^3$  solution to the Benjamin-Ono equation. 
Then the linearized equation \eqref{lin} is well-posed in $H^1$, with uniform 
bounds
\begin{equation}
\| v\|_{C(0,1;H^1)} \lesssim \|v_0\|_{H^1}.
\end{equation}
\end{lemma}
Compared with the main theorem, here the implicit constant is allowed to depend on the 
$H^3$ norm of $\phi$.

\begin{proof}
  The lemma is proved using energy estimates. We begin with the easier
  $L^2$ well-posedness.  On one hand, for solutions for \eqref{lin}
we have the bound 
\[
\frac{d}{dt} \| v\|_{L^2}^2 = \int_\R v \partial_x (\phi v)\, dx =  \frac12 \int_\R v^2 \partial_x \phi  \, dx  \lesssim \|\phi_x\|_{L^\infty}  \| v\|_{L^2}^2 ,
\]
which by Gronwall's inequality shows that 
\[
\| v\|_{L^\infty _tL^2_x} \lesssim \|v_0\|_{L^2_x} ,
\]
thereby proving uniqueness. On the other hand, for the (backward) adjoint problem
\begin{equation}
(\partial_t +H \partial^2_x) w  =   \phi \partial_x w, \qquad w(1) = w_1
\end{equation}
we similarly have 
\[
\| w\|_{L^\infty _tL^2_x} \lesssim \|w_1\|_{L^2_x} ,
\]
which proves existence for the direct problem.

To establish $H^1$ well-posedness in a similar manner we rewrite our
evolution as a system for $(v,v_1:= \partial_x v)$,
\[
\left\{
\begin{array}{l} 
(\partial_t +H \partial^2_x) v  =   \partial_x  ( \phi v ), \cr
(\partial_t +H \partial^2_x) v_1  =   \partial_x  ( \phi v_1 ) + \phi_x v_1 + \phi_{xx} v.
\end{array}
\right.
\]
An argument similar to the  above one shows that this system is also $L^2$ well-posed.
Further, if initially we have $v_1 = v_x$ then this condition is easily propagated in time.
This concludes the proof of the lemma.
\end{proof}

Given the last Lemma ~\ref{l:dippi}, in order to prove Theorem ~\ref{t:liniarizare}, it suffices to show that the 
$H^1$ solutions $v$ given by the Lemma~\ref{l:dippi} satisfy the bound \eqref{S-lin}. It is convenient 
in effect to prove stronger bounds. To state them we assume that $\| v(0)\|_{H^{-\frac12}} \leq 1$,
and consider a frequency envelope  $d_k$ for $v(0)$ in $H^{-\frac12}$.  Without any restriction in generality 
we may assume that $c_k \leq d_k$, where $c_k$ represents an $L^2$ frequency envelope for $\phi(0)$ 
as in the previous section. With these notations, we aim to  prove that the dyadic pieces $v_k$ of $v$ 
satisfy the Strichartz estimates
\[
\| v_k \|_{S^0} \lesssim 2^{\frac{k}{2}} d_k,
\]
as well as the bilinear $L^2$ estimates 
\[
\|L(v_j, \phi_k)\|_{L^2} \lesssim \epsilon d_j c_k 2^{\frac{j}{2}}\cdot 2^{-\frac{\min\{j,k\}}{2}}.
\]
Again, here we allow for $j = k$ under a $2^{k-4}$ frequency separation condition.
Since $v$ is already in $H^1$ and $\phi$ is in $H^3$, a continuity argument shows that it suffices to make
the bootstrap assumptions
\begin{equation}
\label{boot-lin}
\| v_k\|_{S^0} \leq C 2^{\frac{k}{2}} d_k,
\end{equation}
\begin{equation}
\label{boot-lin-bi}
\sup_{y\in \R}\| v_j T_y\phi_k\|_{L^2} \lesssim C\epsilon d_j c_k 2^{\frac{j}{2}} 2^{-\frac{\min\{j,k\}}{2}} ,\quad j\neq k,
\end{equation}
and prove that
\begin{equation}
\label{get-lin}
\| v_k\|_{S^0} \lesssim ( 1 + \epsilon C) 2^{\frac{k}{2}} d_k,
\end{equation}
respectively 
\begin{equation}
\label{get-lin-bi}
\sup_{y\in \R}\| v_j T_y\phi_k\|_{L^2} \lesssim  \epsilon ( 1 + \epsilon C)  d_j c_k 2^{\frac{j}{2}} 2^{-\frac{\min\{j,k\}}{2}} , \quad j\neq k .
\end{equation}

We proceed in the same manner as for the nonlinear equation, rewriting the 
linearized equation in paradifferential form as 
\begin{equation}
\label{linearization}
A_{BO}^{k,+} v_k^+ = i P_k^+ \partial_x ( \phi \cdot v)  -  i  \phi_{<k}  \partial_x v_k^+
+ \frac12 \partial_x (H+i) \phi_{<k} \cdot  v_k^+.
\end{equation}
Here, in a similar manner as before, we isolate the case $k = 0$, where no paradifferential terms 
are kept on the left.

The next step is to use a normal form transformation to eliminate
quadratic terms on the right, and replace them by cubic terms. The
difference with respect to the prior computation is that here we leave
certain quadratic terms on the right, because their corresponding
normal form correction would be too singular. To understand why this
is so we begin with a formal computation which is based on our prior
analysis for the main problem. Precisely, the normal form which
eliminates the full quadratic nonlinearity in the linearized equation
(i.e. the first term on the right in \eqref{linearization}) is
obtained by linearizing the normal for for the full equation, and is
given by
\begin{equation}
\label{part1}
-\frac{1}{4}P_{k} ^{+}\left[ Hv\cdot \partial_{x}^{-1}\phi\right] -\frac{1}{4}P_{k}^{+}H\left[ v \cdot \partial^{-1}_x\phi\right]
-\frac{1}{4}P_{k} ^{+}\left[ H\phi\cdot \partial_{x}^{-1}v\right]-\frac{1}{4}P_{k}^{+}H\left[ \phi \cdot \partial^{-1}_xv\right].
\end{equation}
On the other hand, the correction which eliminates the paradifferential component (i.e., last two terms in \eqref{linearization}) is given by
\begin{equation}
\label{part2}
\frac{1}{2}HP_{k}^{+}v \cdot \partial_{x}^{-1}P_{<k}\phi ,
\end{equation}
which corresponds to an asymmetric version of the first term in $B_k$
in \eqref{partial nft}. Thus, the full normal form correction for the
right hand side of the equation \eqref{linearization} is \eqref{part1}
$+$ \eqref{part2}.  The term in \eqref{part2} together with the last
two entries in \eqref{part1} yield a commutator structure as in $B_k$
in the previous section. To obtain a similar commutator structure for
the first two terms in \eqref{part1} we would need an additional
correction
\begin{equation}
\label{part3}
\frac{1}{2}HP_{k}^{+}\phi \cdot \partial_{x}^{-1} P_{<k} v.
\end{equation}
Precisely, if we add the three expressions above we obtain the linearization of $B_k$, 
\[
\eqref{part1}+\eqref{part2}+\eqref{part3}=2B_{k}(v, \phi),
\]
where $B_k$ stands for the symmetric bilinear form associated to the
quadratic form $B_k$ defined in \eqref{commutator B_k}.
Hence, our desired normal form correction is 
\[
\eqref{part1}+\eqref{part2}=2B_{k}(v, \phi)-\eqref{part3}.
\]
Unfortunately the expression \eqref{part3} contains $\partial_x^{-1}v$
which is ill defined at low frequencies. Unlike in the analysis of the 
main equation in the previous section, here we also have no commutator
structure to compensate. To avoid this problem we
exclude the frequencies $<1$ in $v$ from the \eqref{part3} part of the
normal form correction. Thus, our quadratic normal form correction
will be
\begin{equation}
\label{bilinear nft lin}
\begin{aligned}
B_{k}^{lin}(\phi, v) =&2B_{k}(v, \phi)-\frac{1}{2}HP_{k}^{+}\phi \cdot \partial_{x}^{-1}v_{(0,k)}.
\end{aligned}   
\end{equation}
This serves as a quadratic correction for the full quadratic terms in the right hand side of 
 \eqref{linearization}, except for the term which corresponds to the frequencies of size $O(1)$ 
in $w$, namely the expression 
\[
Q^{2,lin}_k(\phi,v) =  i  v_{0} \partial_x  \phi_k^+ - \frac12 \partial_x (H+i) v_{0} \cdot  \phi_k^+.
\]

Following the same procedure as in the normal form transformation for the full equation we
 denote the first normal form correction in the linearized equation by
\begin{equation}\label{partial-lin-nft}
\tv_k^+ := v_k^+ + 2 B_k^{lin}(\phi,v).
\end{equation}
 The equation for $\tv_k^+$ has the form
\begin{equation}
A_{BO}^{k,+} \tv_k^+ =  Q_k^{3,lin}(\phi,\phi,v) + Q_k^{2,lin}( v_0,\phi_k).
\end{equation}
Here $Q_k^{2,lin}$ is as above, whereas $Q_k^{3,lin}$ contains the linearization 
of $Q_k^3$ plus the extra contribution arising from the second term in $B_k^{lin}$,
namely 
\begin{equation}
Q_k^{3,lin}(\phi,\phi,v) = 3 Q_k^{3}(\phi,\phi,v) +  \frac{i}2 \phi_k^+ P_{(0,k)} (v \phi) +  
  \frac{i}2  P_k^+ \partial_x (\phi^2) \partial_{x}^{-1} v_{(0,k)} .
\end{equation}
Again there is a straightforward adjustment in this analysis for the case $k=0$,
following the model in the previous section. This adds a trivial low frequency 
quadratic term on the right.

Finally, for $k > 0$ we renormalize $\tv_k^+$ to 
\[
w_k := e^{-i \Phi_{<k}} \tv_k^+,
\]
which in turn solves the inhomogeneous Schr\"odinger equation 
\begin{equation}
 \label{conjugare-lin}
 \begin{aligned}
(i\partial_t +\partial^2_x)\, w_k = [Q_k^{2,lin}( \phi_0,v_k^+) + 3 \tilde Q_k^{3,lin}(\phi, \phi, v) + 
\tilde Q_k^{4,lin}(v, \phi, \phi,\phi)] e^{-i\Phi_{<k}},
 \end{aligned}
 \end{equation}
where  
\[
\tilde Q_k^{3,lin}(v, \phi, \phi)= Q_k^{3,lin}(v, \phi, \phi) + \frac14  v_k^+ \left( 2\cdot P_{<k}(\phi^2)  -  
\left( P_{<k}\phi\right) ^2 \right),
\]
and 
\[
\tilde Q_k^{4,lin}(v, \phi, \phi)= Q_k^{3,lin}(v, \phi, \phi) + \frac14  B_{k}^{lin}(v,\phi)  \left( 2\cdot P_{<k}(\phi^2)  -  
\left( P_{<k}\phi\right) ^2 \right).
\]
Our goal is now to estimate the initial data for $w_k$ in $L^2$, and the inhomogeneous term 
in $L^1_t L^2_x$. We begin with the initial data, for which we have

\begin{lemma}
The initial data for $w_k$ satisfies
\begin{equation}\label{wk-data}
\| w_k(0) \|_{L^2} \lesssim 2^{\frac{k}2} d_k.
\end{equation}
\end{lemma}
\begin{proof}
It suffices to prove the similar estimate for $\tv_k$, which in turn reduces to 
estimating $B_k^{lin}(\phi,v)$. The same argument as in the proof of Lemma~\ref{l:invertibila}
yields
\[
\| B_k^{lin}(\phi,v) \|_{L^2} \lesssim  k \epsilon d_k , 
\]
which is stronger than we need.
\end{proof}

Next we consider the inhomogeneous term:
\begin{lemma}
The  inhomogeneous terms in the $w_k$ equation satisfy
\begin{equation}\label{wk-inhom}
  \|  Q_k^{2,lin} \|_{L^1 L^2} + \| \tilde Q_k^{3,lin} \|_{L^1 L^2} +\| \tilde Q_k^{4,lin} \|_{L^1 L^2} \lesssim 
2^{\frac{k}2} C \epsilon  d_k.
\end{equation}
\end{lemma}
\begin{proof}
We begin with $Q^{2,lin}_k$, which is easily estimated in $L^2$ using the bilinear Strichartz estimates 
\eqref{boot-lin-bi} in our bootstrap assumption. 

All terms  in the  cubic part $\tilde Q_k^{3,lin}$ have the form 
$L_k(\phi,\phi,v)$ possibly with an added harmless Hilbert transform,
except for the expression $  P_k^+ \partial_x (\phi^2) \partial_{x}^{-1} v_{(0,k)}$.
 For this we have the bound
\[
\| L_k(\phi,\phi,v)\|_{L^1 L^2} \lesssim 2^{\frac{k}2} C^2 \epsilon^2 d_k .
\]
The proof is identical to the similar argument for the similar bound
in Lemma~\ref{perturbative}; we remark that the only difference occurs
in the case when $v$ has the highest frequency, which is larger than
$2^k$.  

We now consider the remaining expression $  P_k^+ \partial_x (\phi^2) \partial_{x}^{-1} v_{(0,k)}$,
which admits the expansion
\[
 P_k^+ \partial_x (\phi^2) \partial_{x}^{-1} v_{(0,k)} = \sum_{j \in (0,k)} 2^{-j} 2^{k} L_k(\phi_k,\phi_{<k},v_j)
+ \sum_{j \in (0,k)} \sum_{l\geq k} 2^{-j} 2^{k} L_k(\phi_l,\phi_{l},v_j) .
\]
Here we necessarily have two unbalanced frequencies, therefore this expression is estimated 
by a direct application of the bilinear $L^2$ bound plus a Strichartz estimate.

The bound for the quartic term is identical to the one in Lemma~\ref{perturbative}.
\end{proof}

Now we proceed to recover the Strichartz and bilinear $L^2$ bounds. 
In view of the last two Lemmas we do have the Strichartz bounds for $w_k$, and thus for $\tv_k$.
On the other hand for the quadratic correction $B^{lin}_k(\phi,v)$ we have
\[
B_k^{lin}(\phi,v) = 2^{-k} L(\phi_{<k},v_k) + \sum_{j \in (0,k)}  2^{-j} L(v_j,\phi_k) + \sum_{j \geq k} 2^{-j} L(\phi_j,v_j).
\]
Therefore, applying one Strichartz and one Bernstein inequality, we obtain
\[
\|  B_k^{lin}(\phi,v) \|_{S} \lesssim C\epsilon d_k ,
\]
which suffices in order to transfer the Strichartz bounds to $v_k$.

To recover the bilinear $L^2$ bounds we again follow the argument in 
the proof of Theorem~\ref{apriori}. Our starting point is  the bilinear $L^2$ bound 
\[
\| \tilde P_j v_j \cdot \tilde P_k \psi_k\|_{L^2} \lesssim C \epsilon d_j c_k 2^{\frac{j}2} 2^{-\frac{\max\{j,k\}}2} 
\]
which is a consequence of Lemma~\ref{l:bi}. To fix the notations we assume that $j < k$;
the opposite case is similar.  To transfer this bound to $v_j^+ \phi_k^+$ we 
write  
\[
\tilde P_j v_j  \tilde P_k \psi_k -  \phi_j^+ e^{-i \Phi_{<j}} \phi_k^+  e^{-i \Phi_{<k}} =
\tilde P_j v_j  (\tilde P_k \psi_k -  \phi_k^+ e^{-i \Phi_{<k}})+
(\tilde P_j v_j  -   v_j^+ e^{-i \Phi_{<j}})   \phi_k^+  e^{-i \Phi_{<k}} .
\]
For the first term we use the bound \eqref{psi-err} for the second
factor combined with the Strichartz bound for the first factor.  It
remains to consider the second term.  We freely drop the exponential,
and use the commutator result in  Lemma~\ref{commutator} to express the difference in
the second term as
\[
\begin{split}
\tilde P_j w_j  -  v_j^+ e^{-i \Phi_{<j}}  = & \ 
(\tilde P_j-1) (\tv^+ e^{  -i \Psi_{<j}}) + B_j^{lin} (\phi,v) e^{-i \Phi_{<j}} 
\\
= & \ 
[\tilde P_j-1, e^{  -i \Phi_{<j}}] v_j^+ +
(\tilde P_j-1) ( B_j (\phi,v)  e^{  -i \Psi_{<j}})
+  B_j (\phi,v) e^{-i \Phi_{<j}} 
\\ 
= & \ 2^{-j} L(\partial_x   e^{  -i \Phi_{<j}}, \phi_j^+)  + L( B_j (\phi,v), e^{-i \Phi_{<j}}) 
\\
= & \  2^{-j} L( \phi_{<j}, v_j,  e^{  -i \Phi_{<j}}) 
+ 2^{-j} L( v_{<j}, \phi_j,  e^{  -i \Phi_{<j}}) + L( \partial^{-1} v_{(0,j)}, \phi_j,  e^{  -i \Phi_{<j}}) 
\\ & \ + \sum_{l > j} 2^{-l} L(v_l,\phi_l, e^{-i \Phi_{<j}}) .
\end{split}
\] 
Now we multiply this by $\phi_k^+$, and estimate in $L^2$ using our 
bootstrap hypothesis. For $l \neq k$ we can use a bilinear $L^2$ estimate 
combined with an $L^\infty$ bound obtained via Bernstein's inequality.
For $l = k$ we use three Strichartz bounds. The exponential is harmlessly discarded in all cases.
We obtain
\[
\| (\tilde P_j w_j  -   \phi_j^+ e^{-i \Phi_{<j}} ) \phi_k^{+}\|_{L^2}
\lesssim C \epsilon^2 2^{-\frac{k}2} d_j d_k  
\]
which suffices.  The same argument applies when the roles of $j$ and $k$ are interchanged.

\section{ $L^2$ well-posedness for Benjamin-Ono}

Here we prove our main result in Theorem~\ref{thm:lwp}. By scaling we can assume that our initial data satisfies
\begin{equation}\label{small+}
\| \phi_0\|_{L^2} \leq \epsilon \ll 1,
\end{equation}
and prove well-posedness up to time $T = 1$. We know that if in addition $\phi_0 \in H^3$ 
then solutions exist, are unique and satisfy the bounds in Theorem~\ref{apriori}. For $H^3$ data
we  can also use the  bounds for the linearized equation in Theorem~\ref{t:liniarizare} to compare 
two solutions,
\begin{equation}\label{lip}
\| \phi^{(1)} - \phi^{(2)}\|_{S^{-\frac12}} \lesssim \| \phi^{(1)}(0) - \phi^{(2)}(0)\|_{C(0,1;H^{-\frac12})}.
\end{equation}
We call this property \emph{weak Lipschitz dependence on the initial data}.

We next use the above Lipschitz property to construct solutions for $L^2$ data. Given any initial 
data $\phi_0 \in L^2$ satisfying \eqref{small+}, we consider the corresponding regularized data
\[
\phi^{(n)}(0) = P_{<n} \phi_0.
\]
These satisfy uniformly the bound \eqref{small+}, and further they admit a uniform frequency envelope 
$\epsilon c_{k}$ in $L^2$,
\[
\| P_k  \phi^{(n)}(0) \|_{L^2} \leq \epsilon c_k .
\]
By virtue of Theorem ~\ref{apriori}, the corresponding solutions
$\phi^{(k)}$ exist in $[0,1]$, and satisfy the uniform bounds 
\begin{equation}\label{S-unif}
\| P_k  \phi^{(n)} \|_{S} \lesssim  \epsilon c_k .
\end{equation}
On the other hand, the differences satisfy 
\[
\| \phi^{(n)} - \phi^{(m)}\|_{S^{-\frac12}} \lesssim \| \phi^{(1)}(0) - \phi^{(2)}(0)\|_{H^{-\frac12)}} \lesssim 
(2^{-n} + 2^{-m}) \epsilon .
\]
Thus the sequence $\phi^{(n)}$ converges to some function $\phi$ in $S^{-\frac12}$,
\[
\| \phi^{(n)} - \phi\|_{S^{-\frac12}} \lesssim  2^{-n} \epsilon .
\]
In particular we have convergence in $S$ for each dyadic component, therefore
the function $\phi$ inherits the dyadic bounds in \eqref{S-unif},
\begin{equation}\label{fe-l2}
\| P_k \phi\|_{S}   \lesssim  \epsilon c_k.
\end{equation}
This further allows us to prove convergence in $\ell^2 S$. For fixed $k$ we write
\[
\limsup \| \phi^{(n)} - \phi\|_{\ell^2 S} \leq \limsup \| P_{<k} (\phi^{(n)} - \phi)\|_{\ell^2 S}
+ \| P_{\geq k} \phi\|_{\ell^2 S} + \limsup \| P_{\geq k} \phi^{(n)} \|_{\ell^2 S} \leq c_{\geq k} .
\]
Letting $k \to \infty$ we obtain 
\[
\lim \| \phi^{(n)} - \phi\|_{\ell^2 S} = 0.
\]
Finally, this property also implies uniform convergence in $C(0,1;L^2)$; this in turn allows 
us to pass to the limit in the Benjamin-Ono equation, and prove that the limit $\phi$ 
solves the Benjamin-Ono equation in the sense of distributions.

Thus, for each initial data $\phi_0 \in L^2$ we have obtained a weak solution $\phi \in \ell^2 S$,
as the  limit of the solutions with regularized data. Further, this solution satisfies 
the frequency envelope bound \eqref{fe-l2}.

Now we consider the dependence of these weak solutions on the initial data. First of all, the $\ell^2 S$ 
convergence allows us to pass to the limit in \eqref{lip}, therefore \eqref{lip} extends to these weak 
solutions. Finally, we show that these weak solutions depend continuously on the initial data in $L^2$.
To see that, we consider a sequence of data $\phi^{(n)}(0)$ satisfying \eqref{small+} uniformly, 
so that 
\[
\phi^{(n)}(0) \to \phi_0 \qquad \text{ in } L^2 .
\]
Then by the weak Lipschitz dependence we have
\[
\phi^{(n)} \to \phi \text{ in } S^{-\frac12}.
\]

Hence for the corresponding solutions we estimate
\[
\phi^{(n)} - \phi = P_{<k} (\phi^{(n)} -\phi) + P_{\geq k}   \phi^{(n)} - P_{\geq k} \phi  .
\]
Here the first term on the right converges to zero in $\ell^2 S$ as $n \to \infty$ 
by the weak Lipschitz dependence \eqref{small+}, and the last term converges to zero 
as $k \to \infty$ by the frequency envelope bound \eqref{fe-l2}. Hence letting in order 
first $n \to \infty$ then $k \to \infty$ we have
\[
\limsup_{n \to \infty} \| \phi^{(n)} - \phi\|_{\ell^2 S} \leq \|  P_{\geq k} \phi  \|_{\ell^2 S}+
\limsup_{n \to \infty}  \|  P_{\geq k} \phi^{(n)}  \|_{\ell^2 S}
\]
and then
\[
\limsup_{n \to \infty} \| \phi^{(n)} - \phi\|_{\ell^2 S} \leq 
\lim_{k \to \infty} \limsup_{n \to \infty}  \|  P_{\geq k} \phi^{(n)}  \|_{\ell^2 S}.
\]
It remains to show that this last right hand side vanishes. For this we 
use the  frequency envelope bound \eqref{fe-l2} applied to $\phi^{(n)}$ as follows.

Given $\delta > 0$, we have 
\[
\| \phi^{(n)}(0) - \phi_0\|_{L^2} \leq \delta, \qquad n \geq n_\delta.
\]
Suppose $\epsilon c_k$ is an $L^2$ frequency envelope for $\phi_0$,
and $\delta d_k$ is an $L^2$ frequency envelope for $ \phi^{(n)}(0) -
\phi_0$. Here  $d_k$ is a normalized frequency envelope,
which however may depend on $n$. 
Then $\epsilon c_k+ \delta d_k$ is an $L^2$ frequency
envelope for $ \phi^{(n)}(0)$.  Hence by \eqref{fe-l2} we obtain for
$n \geq n_\delta$
\[
\| P_{\geq k} \phi^{(n)}\|_{\ell^2 S} \lesssim \epsilon c_{\leq k} + \delta d_{\leq k} \lesssim \epsilon c_{\leq k} +\delta.
\]
Thus
\[
 \limsup_{n \to \infty}  \|  P_{\geq k} \phi^{(n)}  \|_{\ell^2 S} \lesssim \epsilon c_{\leq k} +\delta,
\]
and letting $k \to \infty$
we have 
\[
\lim_{k \to \infty} \limsup_{n \to \infty}  \|  P_{\geq k} \phi^{(n)}  \|_{\ell^2 S} \lesssim \delta.
\]
But $\delta > 0 $ was arbitrary. Hence
\[
\lim_{k \to \infty} \limsup_{n \to \infty}  \|  P_{\geq k} \phi^{(n)}  \|_{\ell^2 S} = 0,
\]
and the proof of the theorem is concluded.

\section{The scaling conservation law}

As discussed in the previous section, for the linear equation \eqref{bo-lin} with localized data
we can measure the initial data localization with an $x$ weight, and then propagate
this information along the flow using  the following relation:
\[
\| x \psi(0)\|_{L^2} = \| L\psi (t)\|_{L^2} = \| (x - 2tH \partial_x)\psi(t)\|_{L^2}^2.
\]
The question we ask here is whether there is a nonlinear counterpart to that.
To understand this issue we expand
\[
 \| (x - 2tH \partial_x)\phi(t)\|_{L^2}^2 = \int x^2 \phi^2 - 4xt \phi H\phi_x + 4 t^2 \phi_x^2\,  dx \,dt,
\]
where we recognize the linear mass,  momentum and energy densities.

To define the nonlinear counterpart of this we introduce the nonlinear mass, momentum 
and energy densities as 
\[
\begin{split}
m = &\phi^2,
\\
p = & \  \phi H\phi_x  - \frac13 \phi^3 ,
\\
e = &  \phi_x^2 - \frac34 \phi^2 H \phi_x + \frac18 \phi^4 .
\end{split}
\]
Then we set 
\[
G(\phi) = \int  x^2 m - 4xt p + 4 t^2 e \, dx .
\]
For this we claim that the following holds:

\begin{proposition}
\label{p:energy}
Let $\phi$ be a solution to the Benjamin-Ono equation for which 
the initial data satisfies $\phi_0 \in H^2$, $x\phi_0 \in L^2$. Then 

a) $L \phi \in C_{loc} (\R; L^2(\R))$.

b) The expression $G(\phi)$ is conserved along the flow.

c)  We have the representation
\begin{equation}
G(\phi) = \| \L \phi\|_{L^2}^2
\end{equation}
where
\begin{equation}
\L \phi = x\phi-  2t ( H \phi_x - \frac18(3\phi^2 - (H\phi)^2).
\end{equation}
\end{proposition}

Here one can view the expression $\L \phi$ as a normal form correction to $L \phi$.
While such a correction is perhaps expected to exist, what is remarkable is that it is both nonsingular 
and exactly conserved.

\begin{proof}
a) We first show that the solution $\phi$ satisfies
\begin{equation}\label{xphi}
\| x \phi(t)\|_{L^2} \lesssim_{\phi_0} \langle t \rangle .
\end{equation}
For this we truncate the weight to $x_R$, which is chosen 
to be a smooth function which equals $x$ for $|x| < R/2$ and $R$ for $|x| > R$.
Then we establish the uniform bound
\begin{equation}\label{xrphi}
\frac{d}{dt} \| x_R \phi\|_{L^2}^2 \lesssim_{\phi_0}   1 + \|x_R \phi\|_{L^2}.
\end{equation}
Indeed, we have
\[
\begin{split}
\frac{d}{dt}  \| x_R \phi\|_{L^2}^2 = & \ \int_\R  x_R^2 \phi (-H \partial_x^2 \phi +  \phi \phi_x) \, dx
\\ 
= & \ \ \int_\R  x_R^2 \phi_x  H  \phi_x \, dx  +  \int_\R  2 x_R x'_R (\phi H \phi_x - \frac13  \phi^3)  \, dx
\\
= & \ \ \int_\R  x_R \phi_x  [x_R,H]  \phi_x \, dx  +  \int_\R  2 x_R x'_R (\phi H \phi_x - \frac13  \phi^3)\,  dx
\\
= & \ \ \int_\R  - x'_R \phi  [x_R,H]  \phi_x  - x_R \phi \partial_x [x_R,H]  \phi_x \, dx  +  \int_\R  2 x_R x'_R (\phi H \phi_x - \frac13  \phi^3) \, dx.
\end{split}
\]
Then it suffices to establish the commutator bounds
\[
\| [x_R, H] \partial_x\|_{L^2 \to L^2} \lesssim 1, \qquad  \| \partial_x [x_R, H] \|_{L^2 \to L^2} .
\lesssim 1
\]
But these  are both standard Coifman-Meyer estimates,
which require only $x_R' \in BMO$.

Combining \eqref{xphi} with the uniform $H^1$ bound, we obtain 
\[
\| L \phi\|_{L^2} \lesssim_{\phi_0} \langle t \rangle.
\]
To establish the continuity in time of $L\phi$, we write the evolution equation
\[
( \partial_t + H \partial_x^2 ) L \phi = L \phi \phi_x + H \phi_x \phi_x ,
\]
and observe that this equation is strongly well-posed in $L^2$.

b) Integrating by parts we write
\[
\frac{d}{dt} G(\phi) = \int_\R  x^2 (m_t + 2p_x) - 4xt (p_t +2e_x) \, dx .
\]
It remains to show that the two terms above vanish. For the first we compute
\[
\begin{split}
m_t + 2p_x = & \ - 2 \phi H \phi_{xx} + 2 \phi^2 \phi_x + 2(\phi H \phi_x)_x -    2 \phi^2 \phi_x = 2 \phi_x H \phi_x.
\end{split}
\]
Integrating, we can commute in the $x$ to get
\[
\int x^2 (m_t + 2p_x) \, dx = 2  \int x^2 \phi_x H \phi_x \, dx = \int x \phi_x H(x \phi_x) \, dx= 0
\]
using the antisymmetry of $H$.

For the second term we write
\[
\begin{split}
p_t + 2 e_x = & \ - H \phi_{xx} H \phi_x + \phi \phi_{xxx} + \phi \phi_x H\phi_x + \phi H(\phi \phi_x)_x 
+ \phi^2 H \phi_{xx} - \phi^3 \phi_x 
\\ & \ + 4 \phi_x \phi_{xx} - 3 \phi \phi_x H\phi_x - \frac32 \phi^2 H \phi_{xx} + \phi^3 \phi_x
\\ =  & \  \partial_x(- \frac12 (H \phi_x)^2 + \frac32 \phi_x^2 + \phi \phi_{xx})  + \partial_x ( \phi H(\phi \phi_x) - \frac12 \phi^2 H(\phi_x)) 
\\
& \  - \phi_x H(\phi \phi_x) - \phi \phi_x H \phi_x .
\end{split}
\]
Integrating by parts we have 
\[
\begin{split}
\int x(p_t + 2 e_x) \, dx = & \ - \int - \frac12 (H \phi_x)^2 + \frac32 \phi_x^2 + \phi \phi_{xx} + \phi H(\phi \phi_x) - \frac12 
\phi^2 H(\phi_x) \, dx \\ & \ - \int x( \phi_x H(\phi \phi_x) + \phi \phi_x H \phi_x)\, dx.
\end{split}
\]
To get zero in the first integral we integrate by parts and use the antisymmetry of $H$ together with $H^2 =- I$.
In the second integral we can freely commute $x$ under one $H$ and then use the antisymmetry of $H$.

\bigskip

c) We compute the expression 
\[
Err(\phi) = G(\phi) - \int_\R (x\phi-  2t ( H \phi_x - \frac18(3\phi^2 - (H\phi)^2))^2\, dx.
\]
The quadratic terms easily cancel, so we are first left with an $xt$ term,
\[
Err_1(\phi) = \int - 4xt (- \frac13 \phi^3  +  \frac18 \phi( 3\phi^2 - (H\phi)^2) \,dx.
\]
For this to cancel we need 
\[
\int x \phi^3 \,dx =  3 \int x \phi (H\phi)^2\, dx.
\]
Splitting into positive and negative 
frequencies 
\[
\phi = \phi^+ + \phi^-. \qquad H \phi = \frac1i (\phi^+ - \phi^-),
\]
the cross terms cancel and we are left with having to prove that
\[
\int x (\phi^+)^3\, dx =  \int x (\phi^-)^3 \, dx = 0.
\]
where $\phi^- = \overline{\phi^+}$. By density it suffices to establish this for Schwartz functions $\phi$.
Then the Fourier transform of $\phi^+$ is supported in $\R^+$, and is smooth except for a jump at frequency $0$.
It follows that the Fourier transform of $(\phi^+)^3$ is also supported in $\R^+$ but of class $C^{1,1}$ at zero,
i.e. with a second derivative jump. Hence  the Fourier transform of $(\phi^+)^3$  vanishes at zero and the 
conclusion follows.

Secondly, we are left with a $t^2$ term, namely
\[
Err_2(\phi) = \int 4t^2 ( - \frac34 \phi^2 H \phi_x + \frac14( 3\phi^2 - (H\phi)^2) H \phi_x )  
+ 4 t^2(  \frac18 \phi^4 - \frac{1}{64} (3 \phi^2 - (H\phi)^2)^2 ) \, dx .
\]

The first term cancels since we can integrate out the triple $H\phi$ term. For the second we compute
\[
8 \phi^4 - (3 \phi^2 - (H\phi)^2)^2 = - \phi^4 + 6 \phi^2 (H \phi)^2 - (H\phi)^4 = - 2 (\phi^-)^4 - 2(\phi^+)^4,
\]  
which again suffices, by the same argument as in the first case.
\end{proof}

 We further show that this bound naturally extends to $L^2$ solutions:

\begin{proposition}
Let $\phi$ be a solution to the Benjamin-Ono equation whose initial data satisfies
$\phi_0 \in L^2$, $x \phi_0 \in L^2$. Then $\phi$ satisfies the bounds
\begin{equation}\label{bd1}
\| L \phi\|_{L^2} \lesssim_{\phi_0} \langle t \rangle ,
\end{equation}
\begin{equation}\label{bd2}
\| \phi\|_{L^\infty} \lesssim_{\phi_0} t^{-\frac12} \langle t^\frac12 \rangle .
\end{equation}
Furthermore $\L \phi \in C(\R; L^2)$ and has conserved $L^2$ norm. 

\end{proposition}
We remark that both bounds \eqref{bd1} and \eqref{bd2} are sharp, as they 
must apply to solitons.

\begin{proof}
Since the solution to data map is continuous in $L^2$, it suffices 
to prove \eqref{bd1} and \eqref{bd2} for $H^2$ solutions. Then we
a-priori know that $L \phi \in L^2$ and $\phi \in L^\infty$, and we 
can take advantage of the $\|\L \phi\|_{L^2}$ conservation law.
Hence we can use \eqref{prima} to estimate
\[
\| L \phi \|_{L^2} \lesssim \| \L \phi \|_{L^2} + t \| \phi \|_{L^\infty} \|\phi\|_{L^2}
\lesssim  \| \L \phi \|_{L^2} + t^\frac12 \| L \phi \|_{L^2}^\frac12 \|\phi\|_{L^2}^\frac32 ,
\]
which by Cauchy-Schwarz inequality yields 
\[
 \| L \phi \|_{L^2} \lesssim \| \L \phi \|_{L^2} + t \| \phi\|_{L^2}^3.
\]
Now the pointwise bound bound for $\phi$ follows by reapplying \eqref{prima}.

For the last part, we first approximate the initial data $\phi_0$ with $H^2$ data $\phi_{0}^n$ so that 
\[
\|\phi_0^ n - \phi_0\|_{L^2} \to 0, \qquad  \|x(\phi_0^n - \phi_0)\|_{L^2} \to 0.
\]
Then we have $\|\L \phi^n\|_{L^2} \to \| \L \phi(0)\|_{L^2}$. Since $\phi^n \to \phi_0$ in $L^2_{loc}$,
taking weak limits, we obtain
\[
\| \L \phi\|_{L^\infty L^2} = \| \L \phi(0)\|_{L^2}.
\]
Repeating the argument but with initialization at a different time $t$ we similarly obtain
\[
\| \L \phi\|_{L^\infty_t L^2_x} = \| \L \phi(t)\|_{L^2_x}.
\]
Hence $\| \L\phi\|_{L^2}$ is constant in time. Then, the $L^2$ continuity follows from the 
corresponding weak continuity, which in turn follows from the strong $L^2$ continuity of $\phi$.
\end{proof}

\section{The uniform pointwise decay bound}

In this section we establish our main pointwise decay bound for $\phi$,
namely
\begin{equation}\label{want}
\| \phi(t)\|_{L^\infty} + \|H \phi(t)\|_{L^\infty}  \leq C \epsilon \langle t \rangle^{-\frac12}, \qquad |t| \leq e^{\frac{c}{\epsilon}}
\end{equation}
with a large universal constant $C$ and a small universal constant $c$, to be chosen later.
 
Since the Benjamin-Ono equation is well-posed in $L^2$, with
continuous dependence on the initial data, by density it suffices to
prove our assertion under the additional assumption that $\phi_0 \in
H^2$. This guarantees that the norm $\| u(t)\|_{L^\infty}$ is
continuous as a function of time. Then it suffices to establish the
desired conclusion \eqref{want} in any time interval $[0,T]$ under the
additional bootstrap assumption
\begin{equation}\label{boot}
\| \phi(t)\|_{L^\infty} + \|H \phi(t)\|_{L^\infty}  \leq 2C \epsilon \langle t \rangle^{-\frac12}, \qquad |t| \leq T \leq  e^{\frac{c}{\epsilon}}.
\end{equation}

We will combine the above bootstrap assumption with the bounds arising from the following 
conservation laws:
\begin{align}
\label{use01}
\| \phi(t) \|_{L^2} \leq & \ \epsilon,
\\
\label{use02}
\| \L \phi(t) \|_{L^2} \leq & \ \epsilon,
\\
\label{use03}
\int_{-\infty}^\infty \phi dx = c, \qquad &|c| \leq  \epsilon .
\end{align}
We recall that $\L$ is given by
\[
\L \phi = x\phi-  2t \left[  H \phi_x - \frac18(3\phi^2 - (H\phi)^2)\right] .
\]
One difficulty here is that the quadratic term in $\L \phi$ cannot be treated
perturbatively. However, as it turns out, we can take advantage of its structure
in a simple fashion.

As a preliminary step, we establish a bound on the function 
\[
\partial^{-1} \phi (x) := \int_{-\infty}^x \phi(y)\, dy 
\]
as follows:
\begin{equation}\label{intphi}
| \partial^{-1} \phi(x)| \lesssim  C \epsilon + C^2 \epsilon^2 \log \langle t/x\rangle.
\end{equation}

Assume first that $x \leq -\sqrt{t}$. Then we write
\[
\phi = \frac{1}{x} \L(\phi) +  \frac{2t}x H \phi_x  - \frac{t}{4x}( 3 \phi^2 -(H\phi)^2)) .
\]
Integrating by parts, we have 
\[
\partial^{-1} \phi(x) = \frac{2t}x H\phi(x)  + \int_{-\infty}^x \frac{2t}{y^2} H\phi(y) +\frac{1}{x} \L(\phi)   
- \frac{t}{4y}(    3 \phi^2 -(H\phi)^2) \, dy.
\]
For the first two terms we have a straightforward $\dfrac{C\epsilon \sqrt{t}}{|x|}$ bound due to \eqref{boot}. 
For the third term we use \eqref{use02} and the Cauchy-Schwarz inequality.
For the last  integral term we use the $L^2$ bound \eqref{use01} for $x < -t$ and the $L^\infty$ bound \eqref{boot}
for $-t \leq x \leq -\sqrt{t}$ to get a bound of $C^2 \epsilon^2 \log \langle t/x\rangle$.

This gives the desired bound in the region $x \leq - \sqrt{t}$. A
similar argument yields the bound for $x \geq \sqrt{t}$, where in
addition we use the conservation law \eqref{use03} for $\displaystyle\int \phi \, dy$ to connect $\pm
\infty$. Finally, for the inner region $|x| \leq \sqrt{t}$ we use
directly the pointwise bound \eqref{boot} on $\phi$. This concludes the proof of \eqref{intphi}.

\bigskip

Now we return to the pointwise bounds on $\phi$  and $H \phi$. 
Without using any bound for $t$, we will establish the estimate
\begin{equation} \label{point-get}
\|\phi(t)\|_{L^\infty}^2 + \|H\phi(t)\|_{L^\infty}^2 \lesssim \epsilon^2 t^{-1}(1+ C + C^3 \epsilon \log t + C^4 \epsilon^2 \log^2 t).
\end{equation}
In order to retrieve the desired bound \eqref{want} we first choose $C \gg 1$ in order to account 
for the first two terms, and then restrict $t$ to the range $C \epsilon \log t \ll 1$ for the last two terms.
This determines the small constant $c$ in \eqref{want}.

To establish \eqref{point-get} we first use the expression for $\L (\phi)$ to compute
\[
\frac{d}{dx} (|\phi|^2 + |H\phi|^2) = \frac{1}t F_1 + \frac{1}t F_2 + \frac14 F_3,
\]
where 
\[
F_1 = \phi H \L(\phi) + H \phi \L(\phi), \qquad F_2 = x \phi  H \phi - \phi H(x\phi),
\]
\[ 
F_3 =  - \phi H( 3 \phi^2 - (H\phi)^2) + H \phi (  3 \phi^2 - (H\phi)^2).
\]
We will estimate separately the contributions of $F_1$, $F_2$ and $F_3$. For $F_1$ we combine \eqref{use01} and \eqref{use02}
to obtain 
\[
\| F_1 \|_{L^1} \lesssim \epsilon^2,
\]
which suffices. For $F_2$ we commute $x$ with $H$ to rewrite it as 
\[
F_2(x)  = \phi(x) \int_{-\infty}^\infty \phi(y)\, dy,
\]
which we can integrate using \eqref{intphi}.

Finally, for $F_3$ we use the identity 
\[
H ( \phi^2 - (H \phi)^2) = 2 \phi H\phi
\]
to rewrite it as
\[
F_3 = - \phi H(  \phi^2 + (H\phi)^2) - H \phi (   \phi^2 + (H\phi)^2).
\]
This now has a commutator structure, which allows us to write
\[
\int_{-\infty}^{x_0} F_3(x)\, dx =  - \int_{-\infty}^{x_0} \int_{x_0}^\infty \phi(x) \frac{1}{x-y} (  \phi^2 + (H\phi)^2)(y) \, dy\,  dx .
\] 
Here the key feature is that $x$ and $y$ are separated. We now estimate the last integral. We consider several cases:

\medskip

a) If $ |x -y| \lesssim \sqrt{t} $ then direct integration using \eqref{boot} yields a bound of $C^3 \epsilon^3 t^{-1}$.

\medskip

b) If $|x-y| > t$ then we use \eqref{use01} to  bound $\phi^2 +(H\phi)^2$ in $L^1$. Denoting $x_1 = \min\{x_0,y-t\}$,
we are left with an integral of the form
\[
\int_{-\infty}^{x_1} \frac{1}{x-y} \phi(x) \,dx = \frac{1}{ x_1 - y} \partial^{-1}\phi(x_1) - 
\int_{-\infty}^{x_1} \frac{1}{(x-y)^2} \partial^{-1}
\phi(x) \,dx .
\]
As $|x_1 - y| > t$  from \eqref{intphi} we obtain a bound of 
\[
t^{-1}( C\epsilon^{3} + C^2 \epsilon^4 \log t ).
\]

\medskip

c) $x-y \approx r \in [\sqrt{t},t]$. Then we use \eqref{boot} to bound $  \phi^2 + (H\phi)^2 $ in $L^\infty$ 
and argue as in case (b) to obtain a bound of
\[
t^{-1}( C^3\epsilon^{3} + C^4 \epsilon^4 \log t).
\]
Then the dyadic $r$ summation adds another $\log t$ factor.

\section{The elliptic region}

Here we improve the pointwise bound on $\phi$ in the elliptic region 
$x < - \sqrt{t}$. Precisely, we will show that for $t < e^{\frac{c}\epsilon}$ we have 
\begin{equation}\label{point-ell}
|\phi(x)| + |H\phi(x)| \lesssim \epsilon t^{-\frac14} x^{-\frac12}, \qquad x  \geq \sqrt{t}.
 \end{equation}
To prove this we take advantage of the ellipticity of the linear part $x -2t H \partial_x$ of the operator $\L$
in the region $x  \geq \sqrt{t}$. For this linear part we claim the bound
\begin{equation}\label{lin(L)}
\| x \chi \phi\|^2_{L^2} + \| t \chi \phi_x\|_{L^2}^2 \lesssim \| (x -2t H \partial_x) \phi\|_{L^2}^2+ t^{\frac32} \| \phi \|_{L^\infty}^2 
+ t^\frac12 \| \partial^{-1} \phi\|_{L^\infty}^2,
\end{equation}
where $\chi$ is a smooth cutoff function which selects the region $\{x > \sqrt{t} \}$. 

Assuming we have this, using also \eqref{want}, \eqref{use02} and \eqref{intphi} we obtain
\[
\| x \chi \phi\|^2_{L^2} + \| t \chi \phi_x\|_{L^2}^2 \lesssim \epsilon t^{\frac12} +  
t^2 \| \chi (\phi^2 + (H \phi)^2)\|_{L^2}^2 .
\]
We claim that we can dispense with the second term on the right. Indeed, we can easily use \eqref{want} to bound
the $\phi^2$ contribution by
\[
\| \chi \phi^2\|_{L^2} \lesssim \| \phi\|_{L^\infty} \| \chi \phi \|_{L^2} \lesssim \epsilon t^{-1} \|x \chi \phi\|_{L^2} .
\]
The $(H \phi)^2$ contribution is estimated in the same manner, but in addition we also need to bound the commutator
\begin{equation}\label{com(H)}
\| [H,\chi] \phi\|_{L^2} \lesssim \|\phi \|_{L^\infty} + t^{-\frac12} \|\partial^{-1} \phi \|_{L^\infty}.
\end{equation}
Assuming we also have this commutator bound, it follows that
\begin{equation} \label{ell}
\| x \chi\phi\|^2_{L^2} + \| t (\chi \phi)_x\|_{L^2}^2 \lesssim \epsilon t^{\frac12}.
\end{equation}
This directly yields the  desired pointwise bound \eqref{point-ell} for $\phi$. 

Now we prove the  $H \phi$  part of \eqref{point-ell}.  For $x \approx r > t^\frac12$ we decompose 
\[
\phi = \chi_r \phi + (1-\chi_r)\phi,
\]
where $\chi_r$ is a smooth bump function selecting this dyadic region.

For the contribution of the first term we use interpolation to write
\[
\| H(\chi_r \phi)\|_{L^\infty} \lesssim \| \chi_r \phi\|_{L^2}^\frac12 \| \partial_x(\chi_r \phi)\|_{L^2}^\frac12 \lesssim
\epsilon (t^{\frac14} r^{-1})^\frac12 (t^{-\frac34})^\frac12 = \epsilon t^{-\frac14} r^{-\frac12}.
\]
For the second term we use the kernel for the Hilbert transform, 
\[
 H[ (1-\chi_r)\phi] (x)  = \int \frac{1}{x-y}  [(1-\chi_r)\phi] (y)\, dy .
\]
For the contribution of the region $y > t^\frac12$ we use the pointwise bound \eqref{point-ell} on $\phi$ and directly integrate.
For the contribution of the region $y  < t^\frac12$ we integrate by parts and use the bound \eqref{intphi} on $\partial^{-1} \phi$.
This concludes the proof of the $H\phi$ bound in \eqref{point-ell}.

It remains to prove the bounds \eqref{lin(L)} and \eqref{com(H)}. Both are scale invariant in time, so without any restriction in
generality we can assume that $t = 1$.
\bigskip

{\em Proof of \eqref{com(H)}.}
The kernel $K(x,y)$ of $[\chi,H]$ is given by 
\[
K(x,y) = \frac{\chi(x) -\chi(y)}{x-y},
\]
and thus satisfies
\[
(1+|x|+|y|)|K(x,y)| +  (1+|x|+|y|)^2|\nabla_{x,y}K(x,y)| \lesssim 1
\]
Then we write
\[
\int_\R K(x,y) \phi(y) dy = - \int_\R  K_y(x,y) \partial^{-1} \phi(y) dy
\]
and then take absolute values and estimate.

\bigskip

{\em Proof of \eqref{lin(L)}.}
We multiply $(x - 2H \partial_x)  \phi$ by $\chi := \chi_{\geq 1}(x)$, square and integrate. have 
\[
\|\chi (x - 2H \partial_x) \phi\|_{L^2}^2 - \| \chi x \phi\|_{L^2}^2 - 2 \| \chi |x|^\frac12 |D|^\frac12 \phi\|_{L^2}^2
- \| \chi \phi_x\|_{L^2}^2 = \langle ( T_1+T_2) \phi,\phi \rangle
\]
where 
\[
T_1 = |D| \chi^2 |D| + \partial_x \chi^2 \partial_x, \qquad T_2 =   \chi^2 x |D| + |D|\chi^2 x - 2  |D|^\frac12 \chi^2 x  |D|^\frac12.
\]
Then it suffices to show that
\begin{equation}
\label{t12}
|\langle T_{1,2} \phi,\phi \rangle| \lesssim \| \phi\|_{L^\infty}^2 + \| \partial^{-1} \phi\|_{L^\infty}^2.
\end{equation}
To achieve this we estimate the kernels $K_{1,2}$ of $T_{1,2}$. In order to compute the kernels $K_1$ and $K_2$ we observe that both $T_1$ and $T_2$ have
a commutator structure
\begin{equation}
T_1 =\partial_x \left[ \left[ \chi ^2 \, , \, H\right] \, , \, H \right] \partial_x, \quad T_2= \left[ \left[ \vert D\vert ^{\frac{1}{2}}\, , \, \chi^2 \right] \, , \, \vert D\vert ^{\frac{1}{2}}\right] .
\end{equation}
We first consider $T_1$ for which we claim that its kernel $K_1$ satisfies the bound 
\begin{equation}
\label{k1}
|K_1(x,y)| \lesssim \frac{1}{(1+|x|)(1+|y|)(1+|x|+|y|)}.
\end{equation}
This suffices for the estimate \eqref{t12}.  

To prove \eqref{k1} we observe that instead of analyzing the kernel $K_1(x, y)$, we can analyze the kernel $\tilde{K}_1$: 
\[
K_1(x,y)=\partial_x\partial_y\tilde{K}_{1}(x,y),
\] 
where $\tilde{K}_1$ is the corresponding kernel of the commutator $\left[ \left[ \chi ^2 \, , \, H\right]\,, \, H\right] $, and is given by
\[
\tilde{K}_1(x,y)=\int \frac{\chi^2 (x)- \chi^2 (y)}{x-z} \cdot \frac{1}{z-y}-\frac{\chi^2 (z)- \chi^2 (y)}{z-z} \cdot \frac{1}{x-z}\, dz.
\]
We can rewrite $\tilde{K}_1$ using the symmetry $z\rightarrow x+y-z$
\[
\tilde{K}_1(x,y)=\int \frac{\chi^2 (x)+ \chi^2 (y)  -\chi^2(z) -\chi^2 (x+y-z)}{(x-z)(y-z)}\,  dz.
\]

Secondly, in a similar fashion, we compute the kernel $K_2$ of $T_2$, 
\begin{equation}
\label{k2}
K_2(x,y) =\int \frac{\chi^2 (x) +\chi^2 (y)-\chi^2 (x+y-z)-\chi^2(z)  }{\vert x-z\vert ^{\frac{3}{2}} \vert y-z \vert ^{\frac{3}{2}}}\, dz,
\end{equation}
where again the numerator vanishes of order one at $x=z$ and $y=z$.
For this kernel we distinguish two regions: 
\begin{itemize}
\item $\vert x\vert +\vert y\vert  \lesssim 1$; in this region a direct computation shows that the kernel $K_2$ has a mild logarithmic singularity on the diagonal $x=y$,
\[
\vert K_2 (x,y)\vert \leq 1+ \vert \log \vert x-y\vert \vert .
\] 
\item $\vert x\vert +\vert y\vert  \gg 1$; in this region the kernel $K_2$ is smooth and can be shown to satisfy the bound
\[
 |K_2^{low}(x,y)| \lesssim \frac{ (1+ \min\{|x|,|y|\})^\frac12}{  (1+|x|+|y|)^\frac32} .
\]
This does not suffice for the bound \eqref{t12}. However after differentiation it improves to 
\[
|\partial_x \partial_y K_2^{low}(x,y)| \lesssim \frac{1}{ (1+ \min\{|x|,|y|\})^\frac12 (1+|x|+|y|)^\frac52} ,
\]
and that is enough to obtain \eqref{t12}.
\end{itemize}


\begin{thebibliography}{10}
\bibitem{abfs} L. Abdelouhab, J.L. Bona, M. Felland, and J.-C. Saut, Nonlocal models for nonlinear dispersive waves, \emph{physica D}, 40, 360-392, 1989
\bibitem{ad} T.~Alazard, J.M.~ Delort, Global solutions and asymptotic behavior for two dimensional gravity water waves, \emph{Annales Scientifiques de l École Normale Supérieure}, 48(5), 2013
\bibitem{ad1}T.~Alazard, J.M.~ Delort, Sobolev estimates for two dimensional gravity water waves,  \emph{Astérisque}, No. 374 (2015), viii+241 pp. 
\bibitem{benjamin}  T.B. Benjamin, Internal waves of permanent form in fluids of great depth, \emph{J. Fluid Mech.}, 29 (1967), 559-592.
\bibitem{bp} N. Burq and F. Planchon, On well-posedness for the Benjamin-Ono equation, \emph{Math. Ann.}, (2008), Volume 340, Issue 3, pp 497-542.
\bibitem{cm} R. R.~ Coifman and Y.~Meyer, Le double commutateur, \emph{Analyse Harmonique d'Orsay}, 1976
\bibitem{fs} A. S. Fokas,  P.M. Santini,  Bi-Hamiltonian formulation of the Kadomtsev-Petviashvili and Benjamin-Ono equations. \emph{J. Math. Phys.} 29 (1988), no. 3, 604-617
\bibitem{fp} G.~Fonseca, G. ~Ponce. The IVP for the Benjamin-Ono equation in weighted Sobolev spaces, \emph{ Journal of Functional Analysis}, Vol. 260, Issue 2, 15 January 2011,  436-459
\bibitem{fpl} G.~ Fonseca, G.~ Ponce, F.~ Linares, The IVP for the Benjamin-Ono equation in weighted Sobolev spaces II, \emph{Journal of Functional Analysis}, Vol. 262, Issue 5, 1 March 2012, 2031-2049
\bibitem{gusta} Gustafson, Stephen, Takaoka, Hideo; Tsai, Tai-Peng Stability in H1/2 of the sum of K solitons for the Benjamin-Ono equation. \emph{J. Math. Phys.},  50 (2009), no. 1
\bibitem{itg} B. Harrop-Griffiths, M. Ifrim, and D. Tataru, Finite depth gravity water waves in holomorphic coordinates,  to appear in \emph{ Annals of PDE}, \href{http://arxiv.org/abs/1607.02409}{http://arxiv.org/abs/1607.02409}, 2016 
\bibitem{kp1}  B. Harrop-Griffiths, M. Ifrim, and D. Tataru,  The lifespan of small data solutions to the KP-I, to appear \emph{International Mathematics Research Notices}, \href{http://arxiv.org/pdf/1409.4487.pdf}{http://arxiv.org/pdf/1409.4487.pdf}, 2014
\bibitem{herr} S.~ Herr, Well-posedness for equations of Benjamin-Ono type,    \emph{Illinois J. Math.}, Vol. 51, Number 3 (2007), 951-976.
\bibitem{hikk} S.~ Herr, A. Ionescu, D. Alexandru, C. E. Kenig and H. Koch. A para-differential renormalization technique for nonlinear dispersive equations, \emph{Comm. Partial Differential Equations} 35, no. 10, 1827-1875, 2010
\bibitem{BH} J.~K.~Hunter, M.~Ifrim, D.~Tataru, and T.~K.~Wong, Long time solutions for a Burgers-Hilbert equation via a modified energy method, \emph{Proceedings of the AMS}, 2013
\bibitem{hi} J.~K.~Hunter, M.~Ifrim, Enhanced lifespan of smooth solutions of a Burgers-Hilbert equation, \emph{SIAM Journal on Mathematical Analysis}, Volume 44, Issue 3, 2012, 1279-2235.
\bibitem{hit}  J.~K.~Hunter, M.~Ifrim, and  D. Tataru, Two dimensional water waves in holomorphic coordinates,  \emph{Comm. Math. Phys.}, Vol.346(2), pp. 483-552, 2016
\bibitem{itc} M.~Ifrim and D. Tataru, Two dimensional gravity water waves with constant vorticity: I. Cubic lifespan, 2015 \emph{e-print} available at \href{http://arxiv.org/abs/1510.07732}{http://arxiv.org/abs/1510.07732}, 2015 
\bibitem{its} M.~Ifrim and D. Tataru,  The lifespan of small data solutions in two dimensional capillary water waves, 2014 \emph{e-print} available at \href{http://arxiv.org/pdf/1406.5471v2.pdf}{http://arxiv.org/pdf/1406.5471v2.pdf}, to appear in \emph{Archive for Rational Mechanics and Analysis}
\bibitem{itgr}  M.~Ifrim and D. Tataru,  Two dimensional water waves in holomorphic coordinates II: global solutions, \emph{Bull. Soc. Math. France}, Vol. 144(2), pp. 369-394, 2016  
\bibitem{nls} M. ~Ifrim and D. ~Tataru, Global bounds for the cubic nonlinear Schr\"odinger equation (NLS) in one space dimension, \emph{Nonlinearity}, Volume 28, Number 8, 2015.
\bibitem{iorio}] R. J. Iorio, On the Cauchy problem for the Benjamin-Ono equation, \emph{Comm. Partial Differential Equations}, 11(1986), 1031–1081.
\bibitem{ik} A. D. Ionescu and C. E. Kenig, Global well-posedness of the Benjamin-Ono equation in low regularity spaces, \emph{Journal of the American Mathematical Society}, 20 (2007), 753-798.
\bibitem{klm} D. J. Kaup, T. I.  Lakoba, and Y.  Matsuno, Complete integrability of the Benjamin-Ono equation by means of action-angle variables, \emph{ Phys. Lett., A} 238, No.2-3, 123-133 (1998). 
\bibitem{KK} C. E. Kenig and K. D. Koenig, On the  local  well-posedness  of  the  Benjamin-Ono and  modified Benjamin-Ono equations, \emph{Math. Res. Lett.}, 10 (2003), 879-895.
\bibitem{km} C. E. Kenig and Y. Martel ,  Asymptotic stability of solitons for the Benjamin-Ono equation, \emph{    Rev. Mat. Iberoamericana}, Volume 25, Number 3 (2009), 909-970.
\bibitem{klein-saut} C.  Klein, J.-C. Saut, IST Versus PDE: A Comparative Study,  \emph{Fields Institute Communications}, Springer, Vol. 75, 383-449, 2015
\bibitem{KT} H. Koch and N. Tzvetkov, On the local well-posedness of the Benjamin-Ono equation in $H^s(\mathbf{R})$, \emph{IMRN.}, (2003), 1449-1464.
\bibitem{mp} L. Molinet and D.Pilot,  The Cauchy problem for the Benjamin-Ono equation in $L^2$ revisited, \emph{Analysis and PDE}
Vol. 5, No. 2, 2012
\bibitem{MST} L. Molinet, J.C. Saut, and N. Tzvetkov, Ill-posedness  issues  for  the  Benjamin-Ono  and  relatedequations, \emph{SIAM J. Math. Anal.},  33 (2001), 982-988.
\bibitem{ono} H.~Ono. \newblock Algebraic solitary waves in stratified fluids.\emph{ J. Phys. Soc. Japan}, 39:1082-1091, 1975.
\bibitem{ponce} G. Ponce, On  the  global  well-posedness  of  the  Benjamin-Ono  equation, \emph{ Diff. Integral Eq.}, 4 (1991), 527-542.
\bibitem{saut} J.C. Saut, Sur quelques g\'en\'eralisations de l\' ́equation de Korteweg -de Vries, \emph{J. Math. Pures Appl.}, 58 (1979) 21-61.
   \bibitem{shatah} J. ~ Shatah, Normal forms and quadratic nonlinear Klein-Gordon equations, \emph{Comm. Pure Appl. Math.}, Vol. 38, (1985), 685-696.
 \bibitem{tao-m}  Tao, Terence, Multilinear weighted convolution of L2-functions, and applications to nonlinear dispersive equations. \emph{Amer. J. Math.} 123 (2001), no. 5, 839–908. 
\bibitem{tao}  T. ~ Tao, Global well-posedness of the Benjamin-Ono equation in $𝐻^1(\mathbb{R})$,  \emph{J. Hyperbolic Differ. Equ.},  1 (2004), no. 1, 27–49.
\bibitem{tao-c} T. ~ Tao, Global regularity of wave maps II. Small energy in two dimensions, \emph{Commun. Math. Phys.}, 224, 443 – 544 (2001)
\bibitem{tataru}  D.~ Tataru, On global existence and scattering for the wave maps equation, \emph{Amer. J. Math.},  123 (2001), no. 3, 385--423
\bibitem{wu} S.~Wu, Almost global wellposedness of the 2-D full water wave problem, \emph{Invent. Math.}, 177, 2009, 45-135.
\end{thebibliography}
\end{document}